\documentclass[12pt]{amsart}
\usepackage{enumerate,enumitem,amsmath,amssymb, mathrsfs}
\usepackage{xcolor}

\usepackage{latexsym, amssymb,enumerate}
\usepackage{appendix}

\newcommand{\owedge}{\bcw}

\newcommand{\bpf}{\begin{proof}}
\newcommand{\epf}{\end{proof}}
\newcommand{\beq}{\begin{equation}}
\newcommand{\eeq}{\end{equation}}
\newcommand{\beqn}{\begin{eqnarray*}}
\newcommand{\eeqn}{\end{eqnarray*}}
\def\bcw{\mathbin{\bigcirc\mkern-15mu\wedge}}

\newcommand\tr{\mathop{\rm tr}\nolimits}

\newcommand\R{\mathbb{R}}

\def\circledwedge{\setbox0=\hbox{$\bigcirc$}\relax \mathbin {\hbox
to0pt{\raise.5pt\hbox to\wd0{\hfil $\wedge$\hfil}\hss}\box0 }}

\def\R{{\mathfrak R}}

\newtheorem{prop}{Proposition}[section]
\newtheorem{theo}[prop]{Theorem}
\newtheorem{lemm}[prop]{Lemma}
\newtheorem{coro}[prop]{Corollary}

\newtheorem{RK}{Remark}
\def\begeq{\begin{equation}}
\def\endeq{\end{equation}}
\def\p{\partial}

\def\R{\mathbb R}

\def\tr{{\rm tr}}

\def \ds{\displaystyle}

\def\S{\mathbb  {S}}

\def\odot{\setbox0=\hbox{$\bigcirc$}\relax \mathbin {\hbox to0pt{\raise.5pt\hbox to\wd0{\hfil $\wedge$\hfil}\hss}\box0 }}

\numberwithin{equation} {section}
\def\tilde{\widetilde}

\begin{document}

\title[compactness and uniqueness for a class of CCE manifolds]{Perturbation compactness and uniqueness for 
a class of conformally compact Einstein manifolds \\} 
\vskip .1in

\author{Sun-Yung Alice Chang, Yuxin Ge, Xiaoshang Jin and Jie Qing\\}

\address{Department of Mathematics, Princeton University, Princeton, NJ 08544, USA}
\email{chang@math.princeton.edu}

\address{IMT,
Universit\'e Toulouse 3 \\118, route de Narbonne
31062 Toulouse, France}
\email{yge@math.univ-toulouse.fr}

\address{School of mathematics and statistics, Huazhong University of science and technology,
Wuhan, 430074, China}
\email{jinxs@hust.edu.cn}

\address{Mathematics Department, UCSC, 1156 High Street, Santa Cruz, CA 95064 USA }
\email{qing@ucsc.edu}

\dedicatory{\bf In honor of Joel Spruck, with admirations.}

\begin{abstract} In this paper, we establish compactness results for some classes of conformally compact Einstein metrics defined on manifolds of dimension $d\ge 4$. In the special case 
when the manifold is the Euclidean ball with the unit sphere as the conformal infinity, the existence of such class of metrics has been established 
in the earlier work of Graham-Lee \cite{GL}.
As an application of our compactness result, we derive the uniqueness of the Graham-Lee metrics. As a second application, we also derive some gap theorem, or equivalently, some results of non-existence CCE fill-ins.
\end{abstract} 

\thanks{Research of Chang is supported in part by NSF grant DMS-1802285 
and Simons Collaboration Grant}
\thanks{Research of Qing is supported in part by NSF grant DMS-1608782
and Simons Collaboration Grant}

\subjclass[2000]{}


\maketitle

\section{Introduction and Statement of results}\label{An introduction and statement of results}

\subsection{Introduction}

Let $X$ be a smooth manifold of dimension $d$ with $d \geq 3$ with boundary $\partial X$.
A smooth conformally compact metric $g^+$ on $X$ is a Riemannian metric such that $g = r^2 g^+$ extends smoothly to the 
closure $\overline{X}$ for some defining function $r$. 
A defining function $r$ is a smooth 
nonnegative function on the closure $\overline{X}$ such that $\partial X = \{r=0\}$ and the differential $D r \neq 0$ on $\partial X$. A 
conformally compact metric $g^+$ on $X$ is said to be conformally compact Einstein (CCE) if, in addition, 
$$
\operatorname{Ric} [g^+] = - (d-1) g^+.
$$
The most significant feature of CCE manifolds $(X, \ g^+)$ is that the metric $g^+$ canonically determines  the conformal
structure $[\hat g]$ on the boundary  $\partial X$, where $\hat g = g|_{T\partial X}$. $(\partial X, \ [\hat g])$ is called the conformal infinity of the conformally 
compact manifold $(X, \ g^+)$. It is of great interest in both the mathematics and theoretic physics communities to understand the correspondences between 
conformally compact Einstein manifolds $(X, \ g^+)$ and their conformal infinities $(\partial X, \ [\hat g])$,  partially due to the interest of AdS/CFT
correspondence in theoretic physics (cf. Maldacena \cite{Mald-1,Mald-2,Mald} and  Witten \cite{Wi}). 

The project we work on in this paper is to address the compactness issue:  Given a sequence
of CCE manifolds $(X, \{g_i^{+}\})$ with $M = \partial X$ and $\{ g_i \} = \{ r_i^2 g_i^{+} \} $ a sequence of compactified metrics, with
$ h_i = g_i|_{M}$; assuming $\{h_i\}$ forms a compact family of metrics in $M$, when is it true that some representatives ${\bar g_i} \in  [g_i]$
with $\{ {\bar g}_i |_M =h_i\} $ also forms a compact family of metrics in $X$?

We remark that, for a CCE manifold, given any conformal infinity  $(M, h)$, a special defining function $r$, which we call the geodesic defining function, exists 
so that $ | \nabla_{r^2 g^{+}} r | \equiv 1 $ in an asymptotic neighborhood $M \times [0, \epsilon)$ of $M$ with $r^2 g^{+}|_M = h$.     
We also remark that the eventual goal  in the study of the compactness problem is to establish the existence result of conformal fill-ins for 
some classes of Riemannian manifolds as the conformal infinities.

One of the difficulties to address the compactness problem is due to the existence of a non-local term in the asymptotic expansion of the metric near 
the conformal infinity. To see this, we recall   
the asymptotic behavior of the compactified metric $g$ of CCE manifold $(X, g^{+})$ of
dimension $d$, with conformal infinity $(M, [h] )$, which has been worked out earlier (see \cite{G00,FG12}). It turns out that the behavior is a bit different depending on whether the dimension $d$ is even or odd.

When $d$ is even, we have the expansion:
\begin{equation}
\label{expansioneven} 
 g:=r^2 g^{+} = h + g^{(2)} {r}^2 +\cdots(\mbox{even powers})+ g^{(d-1)} r^{d-1} + g^{(d)} {r}^{d} + \cdot \cdot \cdot \cdot 
\end{equation}
on an asymptotic neighborhood of $M \times (0, \epsilon)$, where $r$ denotes the geodesic defining function corresponding to the conformal infinity $(\partial X, h)$. The $g^{(j)}$ are tensors on $M$, and $g^{(d-1)}$ is trace-free with respect
to the metric $ h$. For $j$ even and $0 \le j \le
d-2$, the tensor $g^{(j)}$ is locally formally determined by the conformal representative $h$, but $g^{(d-1)}$ is a non-local term which is not determined by $h$, subject to the trace free condition. 
 
When $d$ is  odd, the analogous expansion is
\begin{equation}
\label{expansionodd} 
 g:=r^2 g^{+} = h + g^{(2)} {r}^2 +\cdots(\mbox{even powers}) + k {r}^{d-1} \log r+ g^{(d-1)} {r}^{d-1} + \cdot \cdot \cdot \cdot , 
\end{equation}
where the $g^{(j)}$ terms are locally determined for $j$ even and $0 \le  j \le 
d-2$, $k$ is locally determined and trace-free, the trace of  $g^{(d-1)}$ is locally
determined, but the trace-free part of $g^{(d-1)}$  is again not determined by $h$. 
We remark that $h$ together with $g^{(d-1)}$ determine the whole asymptotic behavior of $g$ (\cite {FG12,Biquard})  near the conformal infinity. 

A model case of a CCE manifold is the hyperbolic ball $\mathbb{B}^d$ with the Poincar\'e metric
$g_{H}$ with the conformal infinity the standard metric $h_c$ on the unit 
$d-1$ sphere $\mathbb{S}^{d-1}$. In this case, it was proved
by \cite{Q} (see also \cite{Dutta} and later on by \cite{LQS}) that $(\mathbb{B}^d, g_{\mathbb{H}} )$ is the unique CCE manifold with metric $h_c$ on $\mathbb{S}^{d-1}$ as its conformal infinity.

Another class of examples of CCE manifolds was constructed by Graham-Lee \cite{GL}, where they proved that any metric on $\mathbb{S}^{d-1}$ close enough in the $C^{2, \alpha}$ norm to $h_c$ is the conformal infinity of some CCE metric on the Euclidean unit ball $\mathbb{B}^d$ for all $d\geq 4$. 

In an earlier paper \cite{CGQ}, in the special case when the dimension $d=4$, we have established a compactness
result for classes of CCE manifolds and derived as 
a consequence the uniqueness of the CCE extensions
of Graham and Lee for the class of metrics on $\mathbb{S}^{3}$
which are $C^{3, \alpha}$ close to $h_c$ on $\mathbb{S}^3$.
 
The goal of this paper to extend the results in
\cite{CGQ} to all dimensions $d\ge 4$. 

Recall that when $d=4$, in \cite{CG} and \cite{CGQ}, we have considered a special choice of compactified metric $g^{*}=e^{2w}g^+$ defined on a CCE manifold $(X^4, g^{+})$ of dimension four; which we named as Fefferman-Graham (FG)  compactification. This metric is defined by solving the PDE \cite[Theorem 4.1] {FG}:
\begin{equation}
\label{fgbis}
- \Delta_{g^{+}} w = 3  \,\,\,\, on  \,\,\, X^4,
\end{equation}
where $(w-\log r)|_{\p X}=0$.

On a general $d$-dimensional CCE manifold $(X, g^{+})$, when $d>4$, we will consider a choice of the compactified metric $g^{*}$ which is a special case of a general class of metrics named as ``adapted metrics"  in an earlier paper by Case-Chang  \cite[Section 6]{casechang}. 
The metric was defined by solving the Poisson equation
\begin{equation}
\label{fg}
- \Delta_{g^{+}} v-\frac{(d-1)^2-9}{4} v= 0  \,\,\,\, on  \,\,\, X^d,
\end{equation}
with the Dirichlet data the constant function one, $g^{*} : = v^{\frac {4}{d-4}} g^+ $ with $g^{*} |_M = h$, some fixed metric on the conformal infinity of $(X, g^+)$.  It is known that $g^{*}$ has 
free Q-curvature (see \cite{casechang,CYa}, see also the discussion in Lemma \ref{fgm2} in the current paper).

In this paper, we first consider the case when $d$ is even. 
 In this case, it turns out the method of proof 
in \cite{CGQ} for $d=4$ case can be directly generalized. A key property we will use is the existence of the obstruction tensor  \cite{FG12,GH}  when $d$ is even and which vanishes for metrics conformal to Einstein metrics. When $d=4$, this obstruction tensor is the Bach tensor,  and for metrics conformal to Einstein metrics, the Bach tensor vanishes (i.e. they are Bach flat). As we will see in
the proof of Theorem \ref{maintheorem3bis} in  Section \ref{Sect:bdy1} below, the equation of obstruction flat tensor is an
elliptic equation which allows us to derive an
$\varepsilon$-regularity property for the compactified metrics $g^{*}$ (under the assumptions of Theorem \ref{maintheorem3bis}), and this in turn allows us to gain the regularity of the metric. This gain is the key step which allows us to apply a contradiction argument to reach
the compactness result in the statement of Theorem \ref{maintheorem3bis} below.

\begin{theo} \label{maintheorem3bis}
Suppose that $X$ is a smooth oriented $d$-dimensional manifold with compact boundary with $d$ even and $d \ge 4$.  
Let $\{g_i^+\}$ be a set of conformally 
compact Einstein metrics on $X$. 
Assume that the corresponding metrics $\{h_i\}$ at conformal infinity have non-negative
scalar curvature, and have Yamabe constants uniformly bounded from below by some positive constant $C_1$. Assume further that  $\{h_i\}$ forms a compact family in the $C^{k,\gamma}$-Cheeger-Gromov topology on ${\p X}$ with $k\ge d-2$ when $d\ge 6$ and $k\ge 3$ when $d=4$. Then there exists some small $\delta_0>0$ 
such that if either 
\begin{enumerate}
\item[]($1^{\prime}$) \quad  \quad $\int_{X^d} (|W|^{d/2}dvol)[g^+_i] < \delta_0$, \quad \quad \quad \\
or  
\item[]($1^{\prime\prime}) $  \quad \quad  $Y(\partial X, [ h_i]) \geq Y(\mathbb{S}^{d-1}, [g_{\mathbb{S}}]) - \delta_0, $ 
\end{enumerate}
then the set $\{g_i^*\}$ of the adapted metrics (after diffeomorphisms that fix the boundary) 
is compact in the $C^{k,\gamma'}$-Cheeger-Gromov topology for all $0<\gamma'<\gamma$  on $\overline X$. 
\end{theo}

When the dimension $d$ of the manifold $X$ is odd, in general, we would not expect the strong estimate $C^{d-1}$ as in the cases when $d$ is even due to  the $k r^{d-1} \log r$ term in the expansion of the metric $g$ in \ref{expansionodd}. The coefficient $k$ of this term happens to be the obstruction tensor \cite{FG12,GH} defined on the boundary of $X$ and which in general may not vanish. Thus when $d$ is odd, we will apply a different strategy to gain the regularity of the compactified metric $g^{*}$. It turns out this
strategy actually works for all dimensions $d$ under the somewhat stronger regularity assumption $C^6$  on the boundary metrics when $d$ is small. Instead of exploring the property of vanishing of the obstruction tensor of the metric as
in the case when $d$ is even, we will explore the regularity property of its associated Einstein metric $g^{+}$. In order to do so, in Section \ref{Sect:bdy2} below,  we will modify  the gauge fixing techniques developed earlier in the works  \cite{Biquard1, CDLS, GL,Lee1} for Einstein metrics. We first obtain the regularity of the adapted metrics near the neighborhood of the conformal infinity; we next introduce some suitable weighted spaces and apply the functional analytic 
techniques for such spaces to avoid the degeneracy and obtain the $\varepsilon$ regularity of the adapted metric $g^{*}$.

The analysis in Sections 3 and 4 outlined above leads us to our second result below dealing with CCE manifolds of general dimensions $d$.

\begin{theo} \label{maintheorem3}
Suppose that $X$ is a smooth oriented $d$-dimensional manifold with compact boundary and with $d\ge 4$. Let $\{g_i^+\}$ be a set of conformally 
compact Einstein metrics on $X$. 
 Assume that the corresponding metrics $\{h_i \}$ at conformal infinity have non-negative
scalar curvature, and their Yamabe constants uniformly bounded from below by some constant $C_1>0$. Assume further $\{h_i\}$ are compact in  the  $C^6$-Cheeger-Gromov topology on ${\p X}$. Then there exists some small $\delta_0>0$ such that if either ($1^{\prime}$) or ($1^{\prime\prime}$) holds, the set $\{g_i^*\}$ of the adapted metrics (after diffeomorphisms that fix the boundary) 
is compact in the $C^{3,\gamma'}$-Cheeger-Gromov topology for all $0<\gamma'<1$ on ${\overline X}$.  
\end{theo}

\begin{RK}
\begin{enumerate}
\item  The results in Theorems \ref{maintheorem3bis} and \ref{maintheorem3} have been proved earlier in \cite{CGQ} when $d=4$ 
 in the $C^{3,\gamma'}$ topology; we remark that this is not the optimal estimate.  With more work, we could improve the estimate to $C^{2,\gamma'}$ by applying the intermediate Schauder estimates due to Gilbarg-H\"ormander \cite{GHo}.

\item If we assume the set $\{h_i\}$ of metrics on the boundary with non-negative scalar curvature that represent the conformal infinities lies in a given set $\mathcal{C}$ of the $C^{5,\gamma}$-Cheeger-Gromov topology, 
we can obtain the compactness result for $g_i^{*}$  in the $C^{3,\gamma'}$-Cheeger-Gromov topology 
for all  $0<\gamma'<\gamma<1$.

\item  In the statements of Theorems \ref{maintheorem3bis} and \ref{maintheorem3}, if we assume further that $\partial X=\mathbb{S}^{d-1}$, the small constant $\delta_0$ could be chosen independent of the topology of $X$. To see so, in our proof of the 
theorems, we can apply our argument instead on a fixed manifold $X$, on a sequence of CCE metrics $(X_i, g_i^+)$, with $\partial X_i = \mathbb{S}^{d-1} $ for each $i$, but we allow $X_i$ to have different topology. We claim the same blow-up analysis of the Cheeger-Gromov theory in Sections \ref{Sect:bdy1}, \ref{Sect:bdy2} and \ref{Sect:compactness} also apply, and which allows us to reach the same compactness results.
\end{enumerate}
\end{RK}

As an application of Theorem \ref{maintheorem3}, we are able to establish some global uniqueness result for the CCE 
metrics on $X$ with prescribed conformal infinities constructed by Graham-Lee \cite{GL}.

\begin{theo} \label{uniqueness} On $(\S^{d-1}, h_c)$ with $d\ge 4$, there is a small $C^{6}$ neighborhood of $h_c$
such that every metric $h$ in the neighborhood allows  
exactly one conformally compact Einstein metric $g^+$ fill-in on $X$ with $(\S^{d-1}, h)$ as its conformal infinity. Moreover, the topology of $X$ is the same as the Euclidean ball $\mathbb{B}^d$.
\end{theo}

As a direct consequence,  we have the following non-existence of CCE fill-ins result.\\
 
\begin{coro} \label{nonexistence}  Let $X$ be a d-dimensional compact differential manifold with boundary $\p X=\mathbb{S}^{d-1}$ and a metric $h$ be defined on $\p X$, with  $d \ge 4$. There is a constant $\varepsilon >0$, such that if $X$ is not homeomorphic to the unit ball $\mathbb{B}^d$ and $\|h-h_c \|_{C^6}\le \varepsilon$, 
then there does not exist any
CCE fill-in on $X$ with $(S^{d-1}, h)$ as its conformal infinity.
\end{coro}

\begin{RK}We remark
\begin{itemize}
\item In the statement of Corollary \ref{nonexistence}, the constant $\varepsilon>0$ could be chosen independent of the topology of $X$.\\

\item There are many examples of manifolds $X$ which satisfy the assumptions in Corollary \ref{nonexistence}.  For example, let $Y$ be a closed manifold topologically different than the unit sphere $\mathbb{S}^{d}$ and $B_r\subset Y$ be a small closed ball in $Y$.  Then $X:= Y\setminus B_r$ satisfies the assumptions of
Corollary \ref{nonexistence}. 
When the dimension $d=4$, we can also take $Y$ be any closed homology $4$-sphere.

\end{itemize}
\end{RK}

The paper is organized as follows: In Section \ref{Sect:prelim}, we recall some basic ingredients  which will be used later in the proofs of main theorems and list some of their key 
properties, including in particular the estimates of the injectivity radius. 
In Section \ref{Sect:bdy1}, we prove the boundary regularity for $X$ when $d$ is even.
In Section \ref{Sect:bdy2}, we present a different proof for the boundary regularity for all $d$ dimensional CCE manifolds $X$  which works for all 
$d$.
In Section \ref{Sect:compactness}, we establish various compactness results for the adapted metrics and prove Theorems \ref{maintheorem3bis} and \ref{maintheorem3}. In Section \ref{Sect:uniqueness}, we prove Theorem \ref{uniqueness} of the uniqueness of Graham-Lee metrics. In Corollary \ref{gap} we establish as an application some gap phenomenon for classes of conformal invariants.\\

We remark that in the paper, we have provided separate arguments to gain the regularity of the compactified metrics on $X$ in Section \ref{Sect:bdy1} (when $d$
is even) and in Section \ref{Sect:bdy2} (for all $d$). In the rest of the paper i.e. in Sections \ref{An introduction and statement of results}, \ref{Sect:prelim}, \ref{Sect:compactness} and \ref{Sect:uniqueness},  the arguments work for both even or odd $d$. 
 

\section{Preliminaries}\label{Sect:prelim}
\subsection{Basic properties of adapted metrics $g^*$.}\label{Subsect:rho}

Let $v$ be a solution of (\ref{fg}).  We define a class of adapted  metrics $g^*$ by $g^{*} := v^{\frac4{d-4}} g^+ $   when the dimension $d$ is greater than 4. First, we recall some asymptotic properties of $v$.

\begin{lemm} (Case-Chang \cite{casechang}, Chang-R. Yang \cite{CYa}) \label{fgmetric}
Suppose $(X^{d}, g^{+})$ is conformally compact Einstein with conformal infinity $({\p X}, [h])$,  fix $h \in [h ]$ and $r$ its
corresponding geodesic defining function.  Assume  $v$ is a solution of (\ref{fg}),
then $ v $ has the asymptotic behavior 
$$ v=  \, r^{\frac{d-4}{2}}(A + B r^3) $$ near ${\p X}$, 
where $A, B$ are functions even in $r$, such that $A|_{\p X}\equiv 1$.
\end{lemm}
 
This lemma is a special case of the general scattering theory on CCE manifolds as described in Graham-Zworski \cite{GZ}. In below we will describe some properties of this adapted metric $g^{*} $.


\begin{lemm} (Case-Chang \cite[Lemma 6.2]{casechang}) 
\label{fgm2} 
With the same notation as in Lemma \ref{fgmetric}, 
the adapted metric $g^*$ is totally geodesic on boundary with 
the free $Q$-curvature, that is,  $Q_{g^{*}} \equiv 0 .$  
\end{lemm}

The result is a special case of a much more general result in \cite{casechang}. To avoid introducing more notations, here we will present a self-contained proof.

\begin{proof} 
Recall the fourth order Paneitz operator is given by
$$
P_4=(-\triangle)^2 +\delta(4A-\frac{d-2}{2(d-1)} R) \nabla + \frac{d-4}{2} Q_4,
$$
where $A=\frac{1}{d-2}(Ric-\frac{R}{2(d-1)}g)$ denotes the Schouten tensor, $\delta$ is the dual operator of the differential $\nabla$ and $Q_4$ is a fourth order $Q$-curvature.
More precisely, let $\sigma_k(A)$ denote the $k$-th symmetric function of the eigenvalues of $A$ and $Q_4:=-\triangle \sigma_1(A)+4\sigma_2(A)+\frac{d-4}{2}\sigma_1(A)^2$. For a  Einstein metric with $Ric_{g^+}=-(d-1)g^+$, 
and $$P_4[g^+] = (- \Delta_{g^+}-\frac{(d-1)^2-1}{4}) \circ (- \Delta_{g^+} -\frac{(d-1)^2-9}{4}) .$$
Therefore, due to the conformal invariant property of the Paneitz operator, we have 
$$ Q_4[g^{*}]= \frac{2}{d-4} 
P_4[g^{*}] 1 = \frac{2}{d-4} v^{-\frac{d+4}{d-4}}P_4[g^{+}] v=0. $$
We also remark that it follows from the asymptotic behavior of $v$ (Lemma \ref{fgmetric}) that $g^*$ is totally geodesic on boundary since $\ds \frac{\p }{\p \nu}\left(\frac{v^{\frac{2}{d-4}}}{r}\right)=0$ on $M$ where $\nu$ is the normal vector on the boundary.
\end{proof}

We now recall the formula of the Ricci curvature under conformal change of metrics, applying to $g^* = \rho^2 g^+$, we get
$$
Ric[g^+] = Ric[g^*] + (d-2)\rho^{-1} \nabla^2 \rho + (\rho^{-1} \triangle\rho - (d-1)\rho^{-2} |\nabla \rho|^2) g^*.
$$
Thus
$$
R[g^+] = \rho^2 (R[g^*] + \frac{2(d-1)}{\rho} \triangle \rho -\frac{d(d-1)}{\rho^2}  |\nabla \rho |^2).
$$
Applying (\ref{fg}), we get
\beq \label{relation3}
R[g^*] = 2(d-1)\rho^{-2}(1-  |\nabla\rho|^2),
\eeq
which in turn gives
\beq
\label{relation5}
Ric[g^*]=- (d-2)\rho^{-1} \nabla^2 \rho+\frac{4-d}{4(d-1)}R[g^*]g^*,
\eeq
and
\beq
\label{relation4}
R[g^*] = - \frac{4(d-1)}{d+2}\rho^{-1} \triangle \rho.
\eeq
 We now recall another important property of the adapted metrics  $ g^*$ established in an earlier work of Case and Chang \cite[Lemma 4.2]{casechang}).

\begin{lemm} 
\label{lem4.1}
Suppose that $X$ is a smooth $d$-dimensional manifold with boundary $\partial X$ and $g^+$ is a conformally compact Einstein metric on 
$X$ with the conformal infinity $(\p X, [h])$ of nonnegative Yamabe type. Let $g^*=\rho^2 g^+$ be the special class of adapted metric (considered in previous Lemmas) associated with the metric $h$ with the positive scalar curvature in the conformal infinity. Then the scalar curvature $R[g^*]$ is positive in 
$X$. In view of (\ref{relation3}), which implies that
\beq\label{estimate-rho}
\|\nabla\rho \| [g^*] \le 1.
\eeq
\end{lemm}

We will see in Section \ref{Sect:compactness} that property (\ref{estimate-rho}) implies the convergence on compact subsets of a sequences of rescaled adapated metrics, which is one of the key ingredients 
to establish the compactness results in Theorem \ref{maintheorem3bis} and Theorem \ref{maintheorem3}.

\vskip .2in

\subsection{Elliptic estimates for the adapted metrics}
  
Let $R_{ikjl}$, $W_{ikjl}$, $R_{ij}$ and $R$ be Riemann, Weyl, Ricci, Scalar curvature tensors respectively.  We recall on general Riemannian manifold $(X, g)$ of dimension $d$, the  fourth-order Bach tensor $B$ is defined as 
\beq
\label{Bach-eq-0}
B_{ij}:=\frac{1}{d-3}\nabla^k\nabla^l W_{ikjl}+\frac{1}{d-2} W_{ikjl} R^{kl}.
\eeq
Recall also the Cotton tensor $\mathcal{C}$ is defined as 
\beq
\label{Cotten-0}
\mathcal{C}_{jik} = A_{ji,k} - A_{jk,i},
\eeq
where $A$ is the Schouten tensor.
Recall also a relation between the divergence of Weyl tensor to the Cotton tensor, namely
\beq
\label{Bach-Cotton}
\nabla^l W_{ikjl}=(d-3) \mathcal{C}_{jik}.
\eeq
Applying this relation (\ref{Bach-Cotton}), we can write the  Bach tensor into the following   form:
\beq
\label{Bach-eq-1}
(d-2)B_{ij}=\Delta R_{ij} - \frac{d-2}{2(d-1)} \nabla_i\nabla_j R -\frac{1}{2(d-1)}\triangle R g_{ij}+ Q_1(Rm),
\eeq
where $Q_1(Rm)$ is the quadratic term  on  Riemann curvature tensor
$$
 Q_1(Rm) := 2 W_{ikjl}R^{kl}-\frac{d}{d-2}{R_i}^kR_{jk}+\frac{d}{(d-1)(d-2)} R R_{ij}+(\frac{1}{d-2}R_{kl}R^{kl}-\frac{R^2}{{(d-1)(d-2)} })g_{ij}.
$$
Thanks to the second Bianchi identity for the Weyl tensor (see for instance \cite{CG,CGY}), we have
$$
-W_{ikjl,mm} - W_{jlmi,km} + W_{mkjl,im} = \Psi_{ikjl},
$$
where $\Psi_{ikjl}:= \mathcal{C}_{lkm,m}g_{ji} + \mathcal{C}_{lmi,m}g_{jk}+  \mathcal{C}_{lik,j}- \mathcal{C}_{jkm,m}g_{li}- \mathcal{C}_{jmi,m}g_{lk}- \mathcal{C}_{jik,l}$. A direct computation leads to rewrite the Bach equation (\ref{Bach-eq-1}) in turns of the Weyl tensor as follows:
\beq\label{Bach-eq-2}
\Delta W_{ikjl} + (d-3)\nabla_l \mathcal{C}_{jki} + (d-3)\nabla_j \mathcal{C}_{lik} + \nabla_i\mathcal{C}_{kjl} +\nabla_k\mathcal{C}_{ilj} : = K_{ikjl}+L_{ikjl},
\eeq
where $K $ is a quadratic of curvatures and $L_{ikjl}:=-B_{ji}g_{kl}-B_{lk}g_{ij}+B_{li}g_{jk}+B_{kj}g_{il}$ is some linear term on the Bach tensors.  


We also recall that the adapted metric $g^*$ which we haven chosen in Section \ref{Subsect:rho} is $Q$-flat, i.e., $Q[g^*]=0$, which can be 
expressed in the following form \cite{P,B}:
\beq \label{Q-eq}
-\triangle R=-\frac{d^3-4d^2+16d-16}{4(d-2)^2(d-1)}R^2+ \frac{4(d-1)}{(d-2)^2}|Ric|^2.
\eeq

In the proof of Theorem \ref{regularityinterior} and Lemma \ref{bdy1} in Section \ref{Sect:bdy1}, we will first estimate the Bach tensor, then incorporate the $Q$-flat property of $g^*$ into the Bach equation (\ref{Bach-eq-1}) to derive estimates of the Ricci curvature of $g^*$ and into the equation (\ref{Bach-eq-2}) to derive estimates of its Weyl curvature. 

In order to estimate the Bach tensor and the Cotton tensor  of $g^*$ in the interior of $X$, as 
$g^{*}$ is conformal to the Einstein metric $g^+$, we can simplify their expressions as 
(2.14) and (2.15) below. 

\begin{lemm}\label{BachCotton}
On $(X, g)$, suppose $\tilde g=e^{2w}g$, we have
\beq
\label{Cottonnew}
 \tilde {\mathcal{C}_{ijk}}:=\mathcal{C}_{ijk}[\tilde g]=\mathcal{C}_{ijk}[g]-g^{ml}W_{kjim}[g]w_l,
\eeq
\beq
\label{Bachnew}
\begin{array}{ll}
\tilde {B_{ij}}:= B_{ij}[\tilde g]=&e^{-2w}B_{ij}[ g]+e^{-2w}(d-4)\langle \nabla w,  {\mathcal{C}}_{i\cdot j}
+ {\mathcal{C}}_{j\cdot i} \rangle_g\\
&+e^{-2w}(d-4) w^{k}w^{l}\ W_{kijl}[g],
\end{array}
\eeq
where $w^{k}= \nabla^k w$ (resp. $w_{k}= \nabla_k w$) is the the contravariant (resp. covariant) derivative of $w$ with respect to the metric $g$. 
\end{lemm}
(\ref{Cottonnew}) and (\ref{Bachnew}) are derived by a routine computation. 

If we apply Lemma \ref{BachCotton} to the adapted metrics $g^*=\rho^2 g^+$, using the fact that both  the Bach tensor and the Cotton tensor for the Einstein metric $g^+$ vanish, and the fact that $W_{jkil}[g^*]=\rho^2 W_{jkil}[g^+]$, we obtain the following formulas for the Bach tensor and Cotton tensor for $g^*$.

\begin{coro}\label{Bachcoro}
Suppose $(X, g^{+})$ is a conformally compact Einstein with adapted metrics $g^*=\rho^2 g^+$. Then,  we have 
\beq
\label{Bachex}
B_{ij}[g^*]=\rho^{-2}(d-4)\rho^k \rho^l W_{ikjl}[g^*]=-(d-4)\rho^{-1}\rho^k \mathcal{C}_{ikj}[g^*],
\eeq
\beq
\label{Cottonex}
 \mathcal{C}_{ijk}[g^*]=\rho^{-1}\rho^l W_{jkil}[g^*],
\eeq
where $\rho^l=\nabla^l \rho$  is the contravariant  derivative of $w$ with respect to the metric $g^*$.
\end{coro}

In the next two lemmas, we will
derive some preliminary estimates of the curvatures of $g^{*}$ and prepare ourselves for the proof of the main results in Section \ref{Sect:bdy1}.

We now recall some basic facts  relating the behavior of the curvatures of $g^*$ on the boundary to that of the curvature of its boundary metric, which we denote by $\hat g$.

We denote $\p_1$ the outward unit boundary normal direction; $\alpha,\beta\in \{2,\cdots, d\}$  the tangential directions
on $M= \p X$.\\

\begin{lemm}\label{bdy}
Suppose $(X, g^{+})$ is conformally compact Einstein with conformal infinity $({\p X}, [h])$. We assume $g^{*}$ is $C^3$. 

Then on the boundary $M=\p X$ we have:
\begin{enumerate}
\item $R=\frac{2(d-1)}{d-2}\hat{R}$;
\item $R_{11}=\frac{d}{2(d-2)}\hat{R}, R_{1\alpha }=0, R_{\alpha \beta}=\frac{d-2}{d-3}\hat{R}_{\alpha \beta}-\frac{1}{2(d-2)(d-3)}\hat{R}g_{\alpha \beta;}$
\item $ W_{\alpha\beta\gamma\delta}=\hat{W}_{\alpha\beta\gamma\delta}$ and Weyl tensor vanishes for all other indices;
\item $ \mathcal{C}_{\alpha\beta\gamma}=\hat{\mathcal{C}}_{\alpha\beta\gamma}$ and Cotton tensor vanishes for all other indices;
\item $\nabla_1 A_{11}=\frac{\nabla_1 R}{2(d-1)}, \nabla_\alpha A_{\beta\gamma}=\hat{\nabla}_\alpha\hat{A}_{\beta\gamma}$ and the first covariant derivatives of Schouten tensor $A$  vanishes for all other indices;
\item $ \nabla_\sigma W_{\alpha\beta\gamma\delta}=\hat{\nabla}_\sigma \hat{W}_{\alpha\beta\gamma\delta},\;\nabla_1 W_{\alpha\beta\gamma1}=-\nabla_1 W_{\alpha\beta1\gamma}=\nabla_1 W_{\gamma1\alpha\beta}=-\nabla_1 W_{1\gamma\alpha\beta}=\hat{C}_{\gamma\alpha\beta}$
and the first covariant derivatives of Wyel tensor $W$  vanishes for all other indices.
\end{enumerate}
\end{lemm}
\vskip .2in

All the identities in Lemma \ref{bdy}  above are straightforward consequence of the Gauss-Codazzi equation and the fact that for the boundary of the adapted metric $g^{*}$ is totally geodesic. Similar results as in the statement has been established before when $d=4$ in the earlier work of \cite[Lemma 2.7]{CG}.  The proof for general dimensions is tedious but similar, which we will place in appendix A.

The next lemma is the iteration process to express the higher order derivatives of the Ricci and Weyl curvatures of $g^*$ on the boundary in term of the curvature of the boundary metric; these formulas will be used in the proof of Lemma \ref{bdy1} in Section \ref{Sect:bdy1}.

\begin{lemm}\label{bdybis1}
Suppose $(X, g^{+})$ is conformally compact Einstein with conformal infinity $({\p X}, [h])$ with  $d\ge 6$. Then, for the $C^{d-1}$ adapted metrics $g^*$,  we have on the boundary $M=\p X$ for the all multi-index $ I = (i_1,\cdots, i_l)$ of the length $|I|:=l\le d-3$ with $1\le i_1,\cdots, i_l\le d$
$$
\nabla_I A=P(\hat\nabla_\gamma{\hat A},  \hat\nabla_\delta \hat{W},  \hat\nabla_\kappa (\nabla_1 R)|_M), \; \nabla_{I} W=P_1(\hat\nabla_\gamma{\hat A},  \hat\nabla_\delta \hat{W},  \hat\nabla_\kappa (\nabla_1 R)|_M),
$$
where $P$ and $P_1$ are  some homogenous polynomials on $(\hat\nabla_\gamma{\hat A},  \hat\nabla_\delta \hat{W},  \hat\nabla_\kappa (\nabla_1 R)|_M)$ with the multi-indices $\gamma,\delta,\kappa$ satisfying  $|\gamma|+|\delta|+|\kappa|\le l$ for each term in the polynomials, each component of $\gamma,\delta,\kappa$ taking values from $2$  to $d$, where $|\cdot|$ designates the length of the multi-indice.
\end{lemm}

\begin{proof}
We prove the result by induction.\\
For $l=0,1$, it follows from Lemma \ref{bdy}.\\
Assume the result is true for $l=r$. When $i_1,\cdots,i_{r+1}$ are not all equal to 1, we could change the order of the covariant derivative  such  that
$$
\nabla_{i} A=\nabla_{i_j} \nabla_{i'}A+P_r(\nabla_m Rm),
$$
where $i_j\neq 1$, $i'$ designates the multi-index removed $i_j$, $|m |\le r$, and $P_r$ involves only the derivatives of  Riemann curvature of the order less than $r$. In such case, the results follow from the induction. 
The proof is similar for the Weyl tensor $W$.

Now we treat the $r+1$ order the normal derivatives $\nabla_1^{(r+1)} A$ and  $\nabla_1^{(r+1)} W$. For this purpose, we study first $\nabla_1^{(r)} \mathcal{C}_{ijk}$. Recall (\ref{Cottonexbis}) and take the $r$ order normal derivatives so that
$$
\label{induction}
\begin{array}{ll}
r \nabla_1^{(r)} \mathcal{C}_{ijk}&= \nabla_1^{(r)}  W_{jki1}+ Q_r(\nabla_m Rm)\\
&=\nabla_1^{(r-1)}  \delta W_{jki\cdot}-\nabla_1^{(r-1)}  \nabla^\beta W_{jki\beta} + Q_r(\nabla_m Rm)\\
&=\nabla_1^{(r-1)}  \delta W_{jki\cdot}- \nabla^\beta \nabla_1^{(r-1)}  W_{jki\beta} + \bar Q_r(\nabla_m Rm)\\
&=(d-3)\nabla_1^{(r)} \mathcal{C}_{ijk} - \nabla^\beta \nabla_1^{(r-1)}  W_{jki\beta} + \bar Q_r(\nabla_m Rm).
\end{array}
$$
Here $Q_r,\bar Q_r$ involves only the derivatives of  Riemann curvature of the order less than $r$ and we use the relations (\ref{relation3}) to (\ref{relation4}) and the assumption in the induction. Therefore, we deduce
$$
(d-3-r)\nabla_1^{(r)} \mathcal{C}_{ijk}=\nabla^\beta \nabla_1^{(r-1)}  W_{jki\beta} - \bar Q_r(\nabla_m Rm),
$$
which yields the desired result for the Cotton tensor $C$. Applying the equations (\ref{Bach-eq-1}) to (\ref{Bach-eq-2}), we obtain 
\beqn
\nabla_1^{(r+1)} A=\nabla_1^{(r-1)}\triangle A-\nabla_1^{(r-1)} \nabla_\beta \nabla^\beta A, \\
\nabla_1^{(r+1)} W=\nabla_1^{(r-1)}\triangle W-\nabla_1^{(r-1)} \nabla_\beta \nabla^\beta W.
\eeqn
Hence, the claim follows.  Thus we have finished the proof of the lemma.
\end{proof}

\vskip .2in

\subsection{Some results in Riemannian geometry for CCE manifolds}
One fundamental tool to achieve compactness results in Riemannian geometry 
is the Cheeger-Gromov convergence theory (see, for example, \cite{CGT, Anderson0} for 
manifolds without boundary, and  \cite{Perales, Kodani, Knox, Wong, AKKLT},  for 
manifolds with boundary). For our purpose, here we recall some basic facts of the Cheeger-Gromov compactness theorem for manifolds with boundary.

\begin{lemm}(\cite[Theorem 3.1]{AKKLT},\cite[Remark 2.7]{CGQ})  \label{AKKLT}
Suppose that ${\mathcal M}(R_0,i_0, h_0, d_0)$ is the set of all compact Riemannian manifolds $(X, g)$ with boundary such that 
$$
\aligned
|Ric_X|\le R_0, & \quad |Ric_{\p X}|\le R_0\\
i_{\text{int}}(X) \ge i_0, & \quad  i_\p (X) \ge 2i_0, \quad i(\p X) \ge i_0,\\
\text{Diam}(X) \le d_0, & \quad \| H \|_{Lip(\p X)} \le h_0,
\endaligned 
$$
where $Ric_{\p X}$ is the Ricci curvature of the boundary, $i(\p X)$ is the injectivity radius of the boundary, $i_{\text{int}}(X)$ is the interior injectivity radius, $i_\partial (X, g) $ is  the boundary injectivity radius and $H$ is the mean 
curvature of the boundary.  Then ${\mathcal M}(R_0,i_0, h_0, d_0)$  is pre-compact in the $C^{1,\alpha}$ Cheeger-Gromov topology for 
any $\alpha\in (0,1)$. Moreover, if the Ricci 
curvatures are bounded in the $C^{k, \alpha}$ norm and the boundaries are all totally geodesic with $k\ge 0$, then one has the pre-compactness in the $C^{k+2, \alpha'}$-Cheeger-Gromov topology with $\alpha'<\alpha$. Furthermore, one has the pre-compactness in the Cheeger-Gromov topology with base points 
when we drop
the assumption on the upper bound of the diameter $\text{Diam}(X)$.
\end{lemm}

Another important tool in  Riemannian geometry is to find criteria to establish the no collapsing phenomenon. In the setting of conformal compact Einstein manifolds, we will achieve this by applying an inequality (\ref{Eq:LQS}) recently discovered
by Li-Qing-Shi (\cite{LQS}, see also \cite{Dutta}. This inequality plays an important role in our proof of Theorem \ref{uniqueness} and Corollary \ref{gap}.)\\

\begin{lemm}\label{Lem:LQS} (Li-Qing-Shi \cite[Theorem 1.3]{LQS}) 
Suppose that $(X^d, g^+)$ is a conformally compact Einstein manifold with its conformal infinity of positive
Yamabe constant $Y(\p X, [h])$. Then, for any $p\in X^d$,
\begin{equation}\label{Eq:LQS}
1\ge \frac {\text{vol}_{g^+}(B(p, r))}{\text{vol}_{g_{\mathbb{H}^d}}(B(r))} \ge \left(\frac {Y(\p X, [h])}
{Y(\mathbb{S}^{d-1}, [g_{\mathbb{S}^{d-1}}])}\right)^\frac {d-1}2 
\end{equation}
\end{lemm}

The last topic in this subsection concerns the injectivity radius estimates for manifolds with boundary. For our purpose we may always assume
that the geometry of the boundary is compact in the Cheeger-Gromov sense.

\begin{lemm} 
\label{bdy-injrad}
Suppose that $(X^d, g^+)$ is a conformally compact Einstein $d$-dimensional manifold with the conformal infinity of Yamabe constant $Y(\p
X, [h]) \ge Y_0 >0$. 
And suppose that the adapted metric $(X^d, g^*)$ has the intrinsic injectivity radius $i(\p X, h) \ge i_o>0$,  and that $ i_\p (X, g^*)\le  i_{\text{int}} (X, g^*)$. Then there is a constant  $C_\p  > 0$, depending on $i_0$ and independent of $Y_0$, such that 
\begin{equation}\label{Eq: bdy-injrad}
\max_X |Rm| (i_\p (X, g^*))^2  + i_\p (X, g^*) \ge C_\p,
\end{equation}
where $Rm$ is Riemann curvature of $g^*$.
\end{lemm}

The same proof in our earlier work \cite[Lemma 3.1]{CGQ} can be modified to establish this lemma, we will skip the proof here.

We get a similar estimate like Lemma \ref{bdy-injrad} for the lower bound of the injectivity radius.

\begin{lemm} 
\label{int-injrad}
Suppose that $(X^d, g^+)$ is a conformally compact Einstein $d$-dimensional manifold with the conformal infinity of Yamabe constant $Y(\p
X, [h]) \ge Y_0 >0$.  And suppose  that $(X^d, g^*)$ is the adapted metric associated with the Yamabe 
metric $h$ on the boundary such that the intrinsic injectivity radius $i(\p X, h) \ge i_o>0$, and that $ i_\p (X, g^*)\ge  i_{\text{int}} (X, g^*)$. Then there is a constant
$C_{\text{int}} > 0$, depending on $Y_0$ and $i_0$, such that 
\begin{equation}\label{Eq: int-injrad}
\max_X |Rm|(i_{\text{int}} (X, g^*))^2  + i_{\text{int}} (X, g^*) \ge C_{\text{int}},
\end{equation}
where $Rm$ is the Riemann curvature of $g^*$.
\end{lemm}

The proof of the above lemma is also similar to the one of \cite[Lemma 3.3]{CGQ}. Here we omit the details.


\subsection{Interior regularity in all dimensions}\label{Sect:Int} 
The interior regularity estimates for CCE manifolds is relatively well known, we will provide the result here just for the sake of completeness. First we recall the definition of harmonic radius on a Riemannian manifold with boundary see \cite{Perales}: 

Assume $(X, \ g)$ is a complete Riemnnian $d$-dimensional manifold with the boundary $\p X$. A local coordinates 
$$
(x_1,x_2,\cdots,x_{d}): B(p, r) \to \Omega \subset \mathbb{R}^d
$$
is said to be harmonic if, 
\begin{itemize}
\item $\triangle x_i=0$ for all $1\le i\le d$ in $B(p, r) \subset X$, when $p\in X$ is in the interior;
\item $\Delta x_i = 0 $ for all $1\le i \le d$ in $B(p, r)\cap X$ and,  on the boundary $B(p, r)\cap \p X$,  $(x_2, x_3, \cdots,x_{d})$ is a harmonic
coordinate in $\p X$ at $p$ while $x_1 = 0$, when $p\in \p X$ is on the boundary.
\end{itemize}
For $\alpha\in (0,1)$ and $ M \in (1,2)$, we define the harmonic radius $r^{1,\alpha}(M)$ to be the biggest number $r$ satisfying the following properties:
\begin{itemize}
\item If $\text{dist}(p, \p X) > r$, there is a harmonic coordinate chart on $B(p, r)$ such that
\beq
\label{relation1}
M^{-2}\delta_{jk} \le g_{jk}(x) \le M^{2}\delta_{jk},
\eeq
and
\beq
\label{relation2}
r^{1+\alpha}\sup|x-y|^{-\alpha}|\p g_{jk}(x)-\p g_{jk}(y)|\le M-1
\eeq
in $\overline{B(p, \frac r2)}$.
\item If $p\in\p X$, there is a boundary harmonic coordinate chart on $B(p, 4r)$ such that \eqref{relation1} and \eqref{relation2}
hold in $\overline{B(p, 2r)}$. 
\end{itemize}

\begin{theo}\label{regularityinterior}
Suppose $(X^{d}, g^{+})$ is a conformally compact Einstein with the $C^{k-2,\gamma}$ adapted metrics  $g^*=\rho^2 g^+$ for $k\ge 2$ and $\gamma\in (0,1)$.  Assume that
\begin{enumerate}
\item Given $M>1$ and $\gamma\in (0,1)$ there exists some $r_0>0$ such that the harmonic radius $r^{1,\gamma}(M)\ge r_0$;
\item there exist positive constants $C,C_1>0$ such that  $\rho(x)\ge C_1$ provided $d_{g^*}(x,\p X)\ge C$;
\end{enumerate}
Then for all $x\in \bar X$ with $d_{g^*}(x,\p X)\ge C$ and for all $r\le r_1:=\min( r_0, C/2)$, we have
\beq
\label{est1int}
\|Ric_{g^*}\|_{C^{k,\gamma}(B(x,r/2))}
\le C(M,\gamma,r_0,C_1, k,\|Rm_{g^*}\|_{C^{k-2,\gamma}(B(x,r_1))};
\eeq
which yields also in harmonic coordinates
\beq
\label{est2int}
\|g^*\|_{C^{k+2,\gamma}(B(x,r/2))}\le C(M,\gamma,r_0,C_1, k, \|g^*\|_{C^{k,\gamma}(B(x,r_1))}).
\eeq
\end{theo}

\begin{proof}

In view of equation (\ref{Q-eq}), it follows from  \cite[Theorem 6.2]{GT}, that the estimate (\ref{est1int}) holds for the scalar curvature since $Rm_{g^*}\in {C^{k-2,\gamma}}$, that is, $R\in C^{k,\gamma}$. Using Lemma \ref{lem4.1} and the formula (\ref{Bachex}), (\ref{relation5}) and (\ref{relation4}),  the Bach tensor $B\in {C^{k-2,\gamma}}(B(x,r))$.  Recall the elliptic system (\ref{Bach-eq-1}). By the classical regularity theory  \cite[Theorem 6.2]{GT}, we derive $Ric_{g^*}\in {C^{k,\gamma}}(B(x,3r/4))$ and the estimate  (\ref{est1int}) holds.  Finally, the estimate (\ref{est2int}) comes from Lemma \ref{AKKLT}. 
\end{proof}
\begin{RK} 
We notice the metric $g^*$ is smooth in the interior.
\end{RK}


\section{Boundary regularity when dimension is even}\label{Sect:bdy1} 
For the interior regularity, we can use the conformal changes for the extended obstruction tensors \cite{G09}, which  in the special case of fourth-order tensor agrees with the Bach tensor. The conformal transformation law for the extended obstruction tensors involves both the conformal factor and its gradient. Hence, the $C^1$-estimates of the conformal factor helps us to handle the regularity away from the boundary. However, to obtain the desired regularity result for the class of adapted metrics on the boundary, in our proof we 
use the fact that  such metrics satisfy some elliptic PDE  for AHE manifolds $X^d$ when the dimension $d$ is even. More precisely,  when
$d$ is even, in \cite{GH, FG12},
they define a conformally invariant obstruction tensor $\mathcal{O}_{ij}$ of the form
\beq
\label{obstruction1}
 \mathcal{O}_{ij}=(\triangle)^{(d-4)/2} \frac{1}{d-3}\nabla^k\nabla^l W_{ikjl}+lots=(\triangle)^{(d-4)/2} B_{ij}+lots,
\eeq
where $B_{ij}$ denotes the fourth-order Bach tensor. The obstruction tensor $\mathcal{O}_{ij}$ vanishes on Einstein metrics hence on any metric conformal to an Einstein metric (e.g \cite{FG12}), thus we have the metric satisfies the elliptic equation
\beq
\label{obstruction1bis}
(\triangle)^{(d-4)/2} B_{ij}+lots=0.
\eeq

For example, in the special case when $d=6$, we have (e.g \cite{FG12})
\beq
\label{obstruction}
\begin{array}{ll}
{B_{ij,k}}^k=&2W_{kijl}B^{kl}+4{A_k}^k B_{ij}-8A^{kl}\mathcal{C}_{(ij)k,l}\\
&+4{\mathcal{C}_{ki}}^l{\mathcal{C}_{lj}}^k-2{\mathcal{C}_i}^{kl}\mathcal{C}_{jkl}-4{A^k}_{k,l}{\mathcal{C}_{ij}}^l+4W_{kijl}{A^k}_mA^{ml}, 
\end{array}
\eeq
where $2\mathcal{C}_{(ij)k}=\mathcal{C}_{ijk}+\mathcal{C}_{jik}$.

Our main result in this section is that the elliptic equation (\ref{obstruction1bis}) 
helps us to gain the regularity of the compactified metric $g^{*}$.
This is a key step which will 
lead to the proof of the
compactness result in Theorem
\ref{maintheorem3bis}.
More precisely we have the following result.

\begin{lemm}\label{bdy1}
Suppose $(X^{d}, g^{+})$ is conformally compact Einstein with positive conformal infinity  $({\p X}, [h])$ with dimension d even and $d\ge 6$.  Assume further that the adapted metric $g^*$ as defined in 
Lemma 2.2 is in the $C^{d-2}$ space satisfying
\begin{enumerate}
\item $\|Rm_{g^*}\|_{C^{d-4}}\le 1$;
\item Given $M>1$ and $\gamma\in (0,1)$ there exists some $r_0>0$ such that the harmonic radius $r^{1,\gamma}(M)\ge r_0$ (The harmonic radius $r^{1,\gamma}(M)$ was introduced in Section \ref{Sect:prelim});
\item $\|h\|_{C^{d-1,\gamma}}\le N$ for some positive constants $N>0$ and $\gamma\in (0,1)$.
\end{enumerate}
Then, there exists some positive constant $C$  such that  for all $x\in \bar X$ and for all $r\le r_0$, we have
\beq
\label{est1}
\|Ric_{g^*}\|_{C^{d-3,\gamma}(B(x,r/2)\cap \bar X)}
\le C(M,\gamma,r_0,d, \|Rm_{g^*}\|_{C^{d-4}(B(x,r_0)\cap \bar X)},  \|h\|_{C^{d-1,\gamma}(B(x,r_0)\cap \p X)}).
\eeq
As a consequence, we have
\beq
\label{est2}
\|g^*\|_{C^{d-1,\gamma}(B(x,r/2)\cap \bar X)}\le C(M,\gamma,r_0,d,  \|Rm_{g^*}\|_{C^{d-4}(B(x,r_0)\cap \bar X)}, \|h\|_{C^{d-1,\gamma}(B(x,r_0)\cap \p X)}).
\eeq
\end{lemm}
\vskip .2in

\begin{proof}
We will use the harmonic coordinate and boundary conditions as stated in Lemma \ref{bdy}.


{To establish the estimates in (\ref{est1}), we observe that in view of equation (\ref{Q-eq}), and that $\|Rm_{g^*}\|_{C^{d-4}}\le 1$ holds, it follows from \cite[Theorem 6.6]{GT}, that the scalar curvature $R$ is in the $ C^{d-3,\gamma}$.}

Applying Lemma \ref{bdybis1}, the restriction of the Schouten tensor $A$ and the Weyl tensor $W$ on the boundary also are in the $C^{d-3,\gamma}$.
 
We now estimate the fourth-order Bach tensor $B$ via  the elliptic system 
of obstruction tensor equations (\ref{obstruction1}) or (\ref{obstruction}).
Thus via the classical regularity theory for the Laplacian operator (\cite[Theorem 8.32]{GT}) that  $B$ is in the $C^{1,\gamma}$ (when $d=6$) or more generally $B$ is in the $C^{d-5,\gamma}$ when $d>6$.

Applying the equation (\ref{Bach-eq-1}) and \cite[Theorem 6.6]{GT} again,  the estimate (\ref{est1}) holds for the Ricci curvature. Thus it follows from Lemma \ref{AKKLT} that estimate in (\ref{est2}) also holds.

\end{proof}

\begin{RK} In Lemma \ref{bdy1}, 
\begin{itemize}
 \item We can similarly obtain high order estimates  of $g^*$, that is, if we assume $h \in C^{k,\gamma}$ with $k\ge d-1$, then $g^*$ is in $C^{k,\gamma}$.
 \end{itemize}
  
\end{RK}


\section{Boundary regularity in all dimensions}\label{Sect:bdy2} 
For conformally compact Einstein manifolds of dimension $d$, when $d$ may not be even, we will now use a different strategy to gain boundary regularity. Namely we will use the method of ``gauged Einstein equations" as in the work of Chru\'sciel-Delay-Lee-Skinner \cite{CDLS} to derive our estimates.   The eventual goal is to gain the regularity of the compactified metric through the choice of a suitable local gauge, from there we gain the regularity
of the Weyl and Cotton tensor near the conformal infinity, which in turn implies the regularity of the fourth-order Bach tensor. 

This section is organized as follows. In Subsection \ref{GaugedEinsteinequation}, 
we present the concept of local gauge for Einstein metric introduced by Biquard \cite{Biquard1}, and derive some $C^{3, \alpha}$ 
regularity of the defining function $\rho$ using the adapted harmonic coordinate  introduced in Lemma \ref{lemmaSection3.1},  from which we derive the closeness of the metric  $ g^+$ related to the approximated metric $t^+$ in Lemma \ref{lemma1Section3.1}.
In Subsection \ref{Localgauge},  we  first establish some uniform estimates for the linearized operator of the gauge condition in Lemma \ref{lemma2bisSection3.1}, then apply the result to prove the existence of some suitable local gauge in the neighborhood of any point on the conformal infinity and derive the estimates for such local gauge in Lemma \ref{lemma2Section3.1}.
In Subsection \ref{varepsilon-regularity}, we  first establish some uniform estimates for the linearized operator with respect to the first variable of  the gauged Einstein functional in Lemma \ref{lemma3bisSection3.1} and derive some $\varepsilon$-regularity result of the gauged metric in Lemma \ref{lemma3Section3.1}, which leads to the  regularity in a neighborhood of any point on  the conformal infinity in Lemma \ref{lemma3bis2Section3.1}. In Subsection \ref{Regularity of  $g^*$}, we apply the estimates in Subsection \ref{varepsilon-regularity} to derive estimates of the Weyl and Cotton tensor of the compactified metric $g^*$ in Lemma \ref{lemma4bisSection3.1}, and after passing such information,
 to obtain the $C^{1, \lambda}$ estimates of $Rm[g^*]$ in a local neighborhood of the conformal infinity in Lemma \ref{lemma4Section3.1},   which is the main result in this section.

\subsection{Gauged Einstein equation}\label{GaugedEinsteinequation}
In \cite{CDLS}, the authors use gauged Einstein equation to study the regularity  issue of $g^+$ up to a diffeomorphism.  Later on Biquard-Herzlich have established  \cite{BiquardHerzlich} a local version of the result. We now briefly describe the set-up of their method, then indicate the modifications to apply their method to our setting.

Let $Z_R(p)$ denote a domain defined by (\ref{B1}) in Appendix B. We consider the nonlinear functional introduced by Biquard \cite{Biquard1} defined on the $d$-dimensional open set $Z_R(p)$ with $p\in\p X$ for two asymptotically hyperbolic metrics $g^+$ and $k^+$.
\beq
\label{gauged Einstein equationbis}
 F(g^+,k^+):=Ric[g^+]+(d-1)g^+- \delta_{g^+}^*(B_{k^+}(g^+)),
\eeq
where $B_{k^+}$ is a linear differential operator on symmetric (0,2) tensor,  which is the infinitesimal version of the harmonicity condition
$$
B_{k^+}(g^+):=\delta_{k^+} g^++\frac12 \mathfrak{d}\tr_{k^+}(g^+).
$$
Here,  $\delta$ denotes the divergence operator of $2$-tensors, $\delta^*$ the symmetrized covariant derivative of the vector field and $\mathfrak{d}$ the exterior derivative. 

We now recall the Lichnerowicz Laplacian $\triangle_L$ on symmetric $2$-tensors given by.
$$
\triangle_L:=\nabla^*\nabla +2\overset{\circ}{Ric}[k^+]-2\overset{\circ}{Rm}[k^+]; 
$$
where 
$$
\overset{\circ}{Ric}[k^+](u)_{ij}=\frac{1}{2}(R_{im}[g^+]{u_j}^m+R_{jm}[k^+]{u_i}^m), 
$$
and
$$
\overset{\circ}{Rm}[k^+](u)_{ij}=R_{imjl}[k^+]u^{ml}.
$$
We have for any asymptotically hyperbolic metrics $k^+$
$$
D_1 F(k^+,k^+)=\frac12(\triangle_L+2(d-1)),
$$
where $D_1$ denotes the differentiation of $F$ with respective to its first variable. 

Recall $k^+$ is an asymptotically hyperbolic (AH) metric on $X$ if $k^+$ is a conformally compact metric on $X$ such that for some compactified metric $k=\varphi^2k^+$ there holds $\|\nabla\varphi\|\equiv 1$ on $\p X$. \\
It is clear that for any asymptotically hyperbolic Einstein metrics $g^+$,
$$
 F(g^+, g^+)=0.
$$
Suppose $(X^{d}, {\p X}, g^{+})$ is a conformally compact Einstein manifold  of  dimension $d\ge 4$ with a conformal infinity  $({\p X}, [h])$ of 
the positive Yamabe type. Assume that our adapted metrics $g^*$ is in the $C^3$ and that we have 
\begin{enumerate}
\item[(H1)] $\|Rm_{g^*}\|_{C^{0}}\le 1$;
\item[(H2)] there exists some $r_0>0$ such that the injectivity radius $i_{\text{int}}(X) \ge r_0,  \quad  i_\p (X) \ge 2r_0, \quad i(\p X) \ge r_0$;
\item[(H3)] $\|h\|_{C^{6}}\le N$ for some positive constants $N>0$.
\end{enumerate}

Hence, we can identify $\{p\in \bar X, \rho(p)\le r_1\}= [0,r_1]\times \p X \subset \{d_{g^*}(p,\p X)\le r_0\}$ for some $r_1>0$ (we could decrease $r_1$ if necessary) as a submanifold with boundary. We consider a $C^{4}$ AH metric on $ [0,r_1/2]\times \p X$ and its compactification:
$$
t^+=\rho^{-2}t, \; t=d\rho^2+h+\rho^2 h^{(2)}, 
$$
where $h^{(2)}$  is the Fefferman-Graham expansion term and intrinsically determined by the boundary metric $h$. Given $2R<r_1/2$, we look for a local diffeomorphism $\Phi:Z_R(p)\to Z_{2R}(p)$ such that $\Phi^*g^+$ solves the gauged Einstein equation in $Z_{R/2}(p)$
\beq
\label{gauged Einstein equation}
 F(\Phi^*g^+,t^+)=0.
 \eeq
We divide  the boundary  $\p Z_R(p):= \p^\infty Z_R(p)\cup  \p^{int} Z_R(p)=(\{\rho=0\}\cap \p  Z_R(p))\cup (\{\rho>0\}\cap \p Z_R(p))$.  
Given a CCE $g^+$ and a regular AH $t^+$ with the same conformal infinity on the local boundary $\Psi_{p,R}(Y_1^\infty) = \partial^\infty Z_R (p)$, 
we try to find a local diffeomorphism $\Phi:Z_R(p)\to Z_{2R}(p)$ such that the gauged condition is satisfied in $Z_{R/2}(p)$ up to the 
diffeomorphism $\Phi$ fixing the boundary $\p^\infty Z_R(p)$, that is
$$
B_{t^+}(\Phi^*g^+)=0\mbox{ in } Z_{R/2}(p).
$$
Here $Y_1$ is defined by (\ref{B2}) in Appendix B and $ Y_1^\infty:=\bar  Y_1\cap \{(0,x')\}$. 
Thus, the gauged Einstein equation (\ref{gauged Einstein equation}) is satisfied in $Z_{R/2}(p)$.
 We know $\rho\in C^{3,\gamma}_{loc}$ for all $\gamma\in (0,1)$  
under the adapted harmonic coordinates for the metric $g^*$.  More precisely, we have the following result.
\begin{lemm} 
\label{lemmaSection3.1} 
Under the assumptions (H1)-(H3), there exists some positive constant $C$ depending on $\gamma$ but independent of  $p\in \p X$ (and the sequence of the metrics) such that for all $p\in \p X$ under the adapted harmonic coordinates
$$
\|\rho\|_{C^{3,\gamma}(Z_{r_1/2}(p))}\le C.
$$
\end{lemm}
\begin{proof}
By the classical elliptic regularity \cite[Theorem 8.33]{GT}, it follows from (\ref{Q-eq}) that the scalar curvature $R\in C^{1,\gamma}_{loc}$ and we have
$$
\|R\|_{C^{1,\gamma}(Z_{r_1}(p))}\le C
$$ 
for all $p\in \p X$ and for all $\gamma\in (0,1)$. Thanks to (\ref{relation4}) and Lemma \ref{lem4.1}, we infer that $ \rho$ is $C^{3,\gamma}$ smooth in  $ Z_{r_1/2}(p)$ under the adapted harmonic coordinates for the metric $g^*$, and 
$$
\|\rho\|_{C^{3,\gamma}(Z_{r_1/2}(p))}\le C.
$$
Therefore, we {established} the desired results.
\end{proof}

\begin{RK} 
In addition, under the assumptions that the metric is in  $C^{2,\gamma}$ space and the scalar curvature is in the $C^{2,\gamma}$ space, there holds $\rho$ is in the $C^{4,\gamma}$ space under the adapted harmonic coordinates for the metric $g^*$. In such case,  we have considered the partial differential derivatives for the $C^{4,\gamma}$ norm, not the covariant derivatives.
\end{RK}

We could identify the neighborhood $\{p\in X |\rho(p)\le r_1/2\}$ of $\p X$ in $X$ as $ [0,r_1/2]\times \p X$. In fact, let $(\theta^2, \cdots, \theta^{d})$ be the harmonic chart of $\p X$. We extend them as harmonic functions $(x^2, \cdots, x^{d})$  in X so that a local chart of  $\{p\in X|  \rho(p)\le r_1/2\}$ could be given by $(\rho,x^2, \cdots, x^{d})$.  In view of Lemma \ref{lemmaSection3.1} , such chart is $C^{3, \gamma}$ compatible with the harmonic coordinates of $X$. 
Thus, recall the $C^{4}$ compacitified AH manifold on $ [0,r_1/2]\times \p X$
\begin{equation}
\label{tequation}
t=d\rho^2+h+\rho^2 h^{(2)},\; t^+=\rho^{-2}t.
\end{equation}
We suppose for $t$,  one has $  i_\p (X) \ge 2r_1, \quad i(\p X) \ge r_1$ ( we could decrease $r_1$ if necessary). We consider $t^+$ as a reference AH metric with the given conformal infinity $h$. For simplicity, we drop the index $i$ for the family of metrics $t_i$ and $t_i^+$ if there is no confusion. Recall near the boundary (in $ [0,r_1/2]\times \p X$), $t_i^+$ is a family  of class $C^{4}$  AH manifolds, and moreover the family of metrics $t_i$ is compact in the $C^{3,\gamma}$-Cheeger-Gromov's topology in $Z_R(p)$ for all $R<r_1/2$, for any $p\in \p X$ and  for  all $\gamma\in (0,1)$.
We define a map $H_v : Z_R(p) \rightarrow X$ by
$$H_v(q)=exp_q(v(q)),$$
where exp denotes the Riemannian exponential map of $t^+.$ It is showed \cite{CDLS} that $H_v$ is differmorphism if $v$ is sufficiently small, and  by \cite[Lemma 4.1]{CDLS}  it extends to a homeomorphism of $Z_R(p)$ fixing the boundary at infinity
pointwise if $v$ is small in the $C_\delta^{1,0}(Z_R(p); TX)$ for $\delta> 0$. \\

Let $\Sigma^2$ denote the bundle of symmetric covariant 2-tensor over $X$. Let  $\varphi_R$ be the cut-off function in $\overline{X}$ such that
$${ \mathrm{supp}}\,\varphi_R\in Z_{R}(p),\ \varphi_R\equiv 1 \ on\ Z_{\frac{R}{2}}(p),\|\varphi_R\|_{C^{k,\lambda}_{k+\lambda}(Z_{R}(p))}\leq C_0R^{-k-\lambda},\, \forall 0\le k\le 2,\,\forall \lambda\in (0,1).$$
We set $g^+_\varphi=t^+ + \varphi (g^+-t^+)$. 

In the steps below, we will try to find a local gauge $H_v$  such that
\beq
\label{Gaugecondition}
B_{t^+}((H_v)^*g^+_\varphi)=0\mbox{ in } Z_{R}(p).
\eeq
The linearized operator on $v$ is $B_{t^+} (\delta_{t^+})^* = \frac 12((\nabla )^*\nabla  -Ric[t^+] )$ which is an isomorphism from $C^{k+2,\lambda}_\delta$ into $C^{k,\lambda}_\delta$, provided $\delta\in (-1,d)$.\\

\begin{lemm} 
\label{lemma1Section3.1} 
Under the assumptions (H1)-(H3), there exists some positive constant $R_0<r_1/2$ independent of $p\in \p X$ (and the sequence of the metrics) such that for all $p\in \p X$, we have  
$$
\begin{array}{ll}
i)& g=t+O(\rho^\lambda)\;\; \forall \lambda\in (0,1);\\
ii) & g^+-t^+\in C^{1, \lambda}_{1+\lambda}\; \forall 0<\lambda<1.  \mbox{ Furthermore, for any }\tilde \lambda\in (\lambda,1), \mbox{there exists some  } \\
&C>0, \mbox{ such that  }\\
&\| g^+-t^+\|_{C^{1,\lambda}_{1+\lambda}(Z_{R_0}(p))}\le CR_0^{\tilde\lambda-\lambda}.
\end{array}
$$
\end{lemm}

\begin{proof}
We consider the boundary harmonic chart $(x^1,\cdots,x^d)$. Let $\phi$ be a chart such that $\phi^{-1}(q)=(\rho(q),x^2,\cdots\cdots,x^d)$. We use the above chart $\phi$. We claim on the boundary $g_{1\gamma}=0$, $\p_1 g_{ij}=0$ for all $\gamma=2,\cdots,d$ and for all $i,j=1,,\cdots,d$. For the first one, we note on the boundary
$$
g(\p_1,\p_\gamma)=\p_\gamma \rho=0,
$$
since $\rho$ vanishes on the boundary $\p X$.  Using (\ref{relation3}), $g_{11}=g(\p_1,\p_1)=1+O(\rho^2)$ so that $\p_1 g_{11}=0$ on $\p X$.\\ 
Again $g(\p_1,\p_1)\equiv1$ on the boundary $\p X$, which yields $g(\nabla_{\p_\gamma}{\p_1},\p_1)=0 $ on the boundary. Together with the fact the boundary is totally geodesic, we get $ \nabla_{\p_\gamma}{\p_1}=0$ on the boundary. On the other hand, by  (\ref{relation5}), we deduce
 $\p_1 g_{1\gamma}=\p_1\p_\gamma \rho=D^2 \rho (\p_1,\p_\gamma)-(\nabla_{\p_1}{\p_\gamma})\rho=D^2 \rho (\p_1,\p_\gamma)=0$ on the boundary. Similarly, it follows from (\ref{relation5}) that $\nabla_\alpha\nabla_\beta\rho=0$ on the boundary so that the Christofell symbols $\Gamma^1_{\alpha\beta}=0$ on the boundary, that is,
$0=\frac12(\p_\beta g_{\alpha 1}+\p_\alpha g_{1\beta}-\p_1 g_{\alpha\beta})=-\frac12\p_1 g_{\alpha\beta}$ since $\p_\beta g_{\alpha 1}=\p_\alpha g_{1\beta}=0$ on the boundary. Thus, we obtain $\p_1 g_{\alpha\beta}=0$ on the boundary and prove the claim. We know the metric $g^*$ is in the $C^{1,\lambda}$ space so that (i) is an immediate result of the above claim. By the Taylor's expansion, the second property comes from the fact $g^*$ is bounded in the $C^{1,\alpha}$ topology for all $\alpha\in (0,1)$.
\end{proof}

\subsection{Local gauge} \label{Localgauge} 
We now return to the step to find a diffeomorphism $H$ which satisfies the equation (\ref{Gaugecondition}). Fixing the boundary $\p^\infty(Z_R(p))$ such that $B_{t^+}H^*g^+=0$ in $Z_R(p)$, which is equivalent to $B_{(H^{-1})^*t^+} g^+=0$ in $H^{-1}(Z_R(p))$.  Given small $R>0$ and $p\in \p X$, let $\Psi_{p,R}: Y_1\subset \mathbb{H}\to Z_R(p)$ be a boundary M\"obius chart (see Appendix B). It follows from \cite[Lemma 6.1]{Lee1} that for any $\lambda\in (0,1)$, we have
$$
\|\Psi_{p,R}^* t^+ - g_{ \mathbb{H}}\|_{3,\lambda;Y_1}\le CR,
$$
where the positive constant $C>0$ is independent of the sequence and the point $p\in \p X$.  We denote $\varphi$ some non-negative smooth cut-off function such that $\varphi\equiv 1$ on $Y_{1/2}$ and $\varphi\equiv 0$ on $\mathbb{H}\setminus Y_1$. We want to glue the metric $\Psi_{p,R}^* t^+$ with the standard hyperbolic metric $g_{\mathbb{H}}$ as follows
$$
t^+_{p,R}=\varphi \Psi_{p,R}^* t^++(1-\varphi)g_{\mathbb{H}}.
$$
There exists some small $\bar R_0>0$ such that the sectional curvature of $t^+_{p,R}$ is negative and $\rho^2 t^+_{p,R}$ is a compact family of AH metrics in the $C^{3,\lambda}$-Cheeger-Gromov topology for all $R\le\bar R_0$, for all $p\in \p X$ and for the sequence (for adapted metrics)  since $\rho^2t^+$ is compact family of AH metrics in the $C^{3,\lambda}$-Cheeger-Gromov topology. We denote $\tilde Z_R(p)$ the related domain of a boundary M\"obius chart for such AH metric $t^+_{p,R}$. We consider  the following mapping $\Psi$.
$$
\begin{array}{lcll}
\Psi:&C_{1+\lambda}^{2,\lambda}(\tilde  Z_{\bar R_0}(p); T\tilde  Z_{\bar R_0}(p))\times C_{1+\lambda}^{1,\lambda}(\tilde Z_{\bar R_0}(p); \Sigma^2) &\to& C_{1+\lambda}^{0,\lambda}(\tilde Z_{\bar R_0}(p); T\tilde Z_{\bar R_0}(p))\times C_{1+\lambda}^{1,\lambda}(\tilde Z_{\bar R_0}(p); \Sigma^2)\\
&(v,w)&\mapsto &(B_{(H_v^{-1})^*t^+_{p,\bar R_0}}(t^+_w),w)
\end{array}
$$
where $t^+_w= t^+_{p,\bar R_0}+  w$.  It is clear that 
$$
D_1 \Psi_1 (0,0)(Y)=B_{t^+_{p,\bar R_0}} ( (\delta_{t^+_{p,\bar R_0}})^*Y)=  \frac 12((\nabla )^*\nabla  -Ric[t^+_{p,\bar R_0}] )Y.
$$
Here $\Psi=(\Psi_1,\Psi_2)$.  It is known \cite[Theorem C]{Lee1} that $D_1 \Psi_1 (0,0):  C_{\delta}^{k,\lambda}\to C_{\delta}^{k-2,\lambda}$ is an isomorphism, provided $\delta\in (-1,d)$. In the following, if there is no confusion, the set $Z_R(p)$ is always related to the metric $t^+$. 

\begin{lemm} 
\label{lemma2bisSection3.1} 
Under the assumptions (H1)-(H3), for any given $\lambda\in (0,1)$, there exist some positive constant $C$ and  some small number $\eta>0$, independent of $\bar R_0$ and $p\in \p X$ (and the sequence of the metrics) such that  $\Psi$ is a $C^1$ mapping, and for  all $(v_i, w_i)\in C_{1+\lambda}^{2,\lambda}(\tilde Z_{\bar R_0}(p); T\tilde Z_{\bar R_0}(p))\times C_{1+\lambda}^{1,\lambda}(\tilde Z_{\bar R_0}(p); \Sigma^2) $ with $\|v_i\|+\|w_i\|\le \eta$  for $i=1,2$, there holds
$$
\begin{array}{ll}
i)& \|D_1 \Psi_1 (0,0)\|+\|(D_1 \Psi_1 (0,0))^{-1}\|\le C,\\
ii)& \| D \Psi_1 (v_1,w_1)-D \Psi_1 (v_2,w_2)\|\le C(\|v_1-v_2\|_{C^{2,\lambda}_{1+\lambda}(\tilde Z_{\bar R_0}(p))}+\|w_1-w_2\|_{C^{1,\lambda}_{1+\lambda}(\tilde Z_{\bar R_0}(p))}).
\end{array}
$$
\end{lemm}
\begin{proof}  If there is no confusion, we denote the metric $t^+_{p,\bar R_0}$ as $t^+$ in the proof. It follows from \cite[Lemma 4.6]{Lee1} that $\|D_1 \Psi_1 (0,0)\|\le C$ for some positive constant $C$ independent of  $\bar R_0$, $p\in \p X$ (and the sequence of the 
metrics) since the family of metrics $t_i$ (resp. $t_i^+$) is compact in the $C^{3,\gamma}$-Cheeger-Gromov topology for all $\gamma\in (0,1)$. \\
Now we prove $\|(D_1 \Psi_1 (0,0))^{-1}\|\le C$ by the contradiction. Recall the sectional curvature  is negative on $\tilde Z_{\bar R_0}(p)$. Therefore, there is no $L^2$ kernel for the linear operator $\frac 12((\nabla )^*\nabla  -Ric[t^+])$. As a consequence,  it follows from \cite[Theorem C]{Lee1} that $D_1 \Psi_1 (0,0):  C_{1+\lambda}^{2,\lambda}\to C_{1+\lambda}^{0,\lambda}$ is an isomorphism since $1+\lambda\in (-1,d)$. We suppose 
$$
\|(D_1 \Psi_1 (0,0))^{-1}[t_i^+]\|\to \infty.
$$ 
Thus, we choose some vector field $v_i\in C_{1+\lambda}^{2,\lambda}(\tilde Z_{\bar R_0}(p); T\tilde Z_{\bar R_0}(p))$ with $\|v_i\|_{C^{2,\lambda}_{1+\lambda}}=1$ and $$ \|(D_1 \Psi_1 (0,0)[t_i^+])v_i\|_{C^{0,\lambda}_{1+\lambda}}\to 0.$$  
Up to a subsequence, $t_i$ converges to $t_\infty$ in the $C^{3,\gamma}$-Cheeger-Gromov topology for all  $\gamma\in (0,1)$. Modulo  a subsequence,  $t_i^+$ converges also to a $C^{3,\gamma}$ AH $t_\infty^+=\rho^{-2}t_\infty$ in the pointed  $C^{3,\gamma}$ -Cheeger-Gromov topology.  On the other hand, by  \cite[Lemma 6.4]{Lee1}, 
$$
\|v_i\|_{C^{2,\lambda}_{1+\lambda}}\le C( \|(D_1 \Psi_1 (0,0)[t_i^+])v_i\|_{C^{0,\lambda}_{1+\lambda}}+ \|v_i\|_{C^{0,0}_{1+\lambda'}}).
$$
where $\lambda'\in (0,\lambda)$ and  $C$ is some positive constant independent of  $\bar R_0$, $p$ and the sequence since $t_i$ is in some compact set in the $C^{3,\gamma}$-Cheeger-Gromov topology. Thus, we have for large $i$
$$
\|v_i\|_{C^{0,0}_{1+\lambda'}}\ge 1/2C.
$$
By the Rellich Lemma \cite[Lemma 3.6]{Lee1}, the mapping $C^{2,\lambda}_{1+\lambda}\hookrightarrow C^{0,0}_{1+\lambda'}$ is a compact embedding so that we infer $\|v_\infty\|_{C^{0,0}_{1+\lambda'}}\ge 1/2C$. On the other hand, we have
$$
(D_1 \Psi_1 (0,0)[t_\infty^+])v_\infty=0.
$$
As above, we have $D_1 \Psi_1 (0,0)[t_\infty^+]:  C_{1+\lambda}^{2,\lambda}\to C_{1+\lambda}^{0,\lambda}$ is an isomorphism so that $v_\infty=0$. This contradiction yields the desired result (i).\\
The proof of the property (ii) is  similar as in \cite[Lemmas 4.2 and 4.4]{CDLS}. 
\end{proof}


\begin{lemm} 
\label{lemma2Section3.1} 
{Under the assumptions (H1)-(H3), for any given $\lambda\in (0,1)$,  there exist some positive constants $C$ and $R_1<\bar R_0/2$  independent of $p\in \p X$ (and the sequence of the metrics), such that for any $p\in \p X$ and for any $R\le R_1$, there exist a small local gauge vector field $\tilde v\in C_{1+\lambda}^{2,\lambda}(\tilde Z_{\bar R_0}(p); T\tilde Z_{\bar R_0}(p))$ which satisfies,
$$
\begin{array}{ll}
i)& H_v^*g^+ \mbox {solves  local gauge for gauged Einstein equation in } Z_{R/2}(p),\\
ii)&\| H_v\|_{C^{2,\lambda}(Z_{\bar R_0/2}(p))}\le C, and \\
iii)&\mbox{for any }\tilde \lambda\in (\lambda,1), \mbox{there exists some  } C_1>0, \mbox{ such that  }
\\
&\|\tilde v\|_{C^{2,\lambda}_{1+\lambda}(\tilde Z_{\bar R_0}(p))}\le C_1 R^{\tilde \lambda-\lambda}{\bar R_0}^{1+\lambda},
\end{array}
$$
where $v=(\Psi_{p,\bar R_0})_*\tilde v$.}
\end{lemm}
\begin{proof} We assume $R$ is small so that $t^+_{p,\bar R_0}=(\Psi_{p,\bar R_0})_*t^+$ on $ (\Psi_{p,\bar R_0})^{-1}(Z_R(p))$.  Let $\varphi:\mathbb{R}\to \mathbb{R}$ be a smooth non-negative cut-off function satisfying $\varphi(s)\equiv 1$ for all $s<1/2$ and $\varphi(s)\equiv 0$ for all $s>1$. We consider $$w_R(x)=\varphi(d_{t}(x,p)/R)(g^+-t^+) .$$ 
Set $\tilde w_R= (\Psi_{p,\bar R_0})^*w_R$. Thanks of Lemma \ref{lemma1Section3.1}, we have $\tilde w_R\in C_{1+\lambda}^{1,\lambda}(\tilde Z_{\bar R_0}(p); \Sigma^2)$ and
$$
\|\tilde w_R\|_{C^{1,\lambda}_{1+\lambda}(\tilde Z_{\bar R_0}(p))}\le CR^{\tilde\lambda-\lambda}{\bar R_0}^{1+\lambda},
$$
with $0<\lambda<\tilde \lambda$ so that $\|\tilde w_R\|_{C^{1,\lambda}_{1+\lambda}(\tilde Z_{\bar R_0}(p))}\to 0$ as $R\to 0$. In view of Lemma \ref{lemma2bisSection3.1}, it follows from the inverse function theorem there exists some small $R_1<\bar R_0/2$ such that for all $R\le R_1$ we have $\tilde v\in C_{1+\lambda}^{2,\lambda}(\tilde Z_{\bar R_0}(p); T\tilde Z_{\bar R_0}(p))$ which solves $\Psi (\tilde v,\tilde w_R)=(0,\tilde w_R)$. Moreover, we estimate
$$
\|\tilde v\|_{C^{2,\lambda}_{1+\lambda}(\tilde Z_{\bar R_0}(p))}\le C R^{\tilde\lambda-\lambda}{\bar R_0}^{1+\lambda}.
$$
To see that, we have $\Psi (0,0)=0$ and $\Psi_1(0,\tilde w_R)=B_{(\Psi_{p,\bar R_0})^*t^+}((\Psi_{p,\bar R_0})^*(t^++ w_R))=(\Psi_{p,\bar R_0})^*B_{t^+}( w_R) $ so that 
$$\|\Psi_1(0,\tilde w_R)\|_{C^{0,\lambda}_{1+\lambda}(\tilde Z_{\bar R_0}(p))}\le C R^{\tilde\lambda-\lambda}{\bar R_0}^{1+\lambda} .
$$ 
Thus, we establish (iii). As a consequence, we obtain
$$
\| v\|_{C^{2,\lambda}_{1+\lambda}( Z_{\bar R_0}(p))}\le C R^{\tilde\lambda-\lambda}.
$$
We have $B_{(H_v^{-1})^*t^+}(t^++w_R)=0$ in $Z_{R_1}(p)$, which yields 
$$
0=(H_v)^*B_{(H_v^{-1})^*t^+}(t^++w_R)=B_{t^+}(H_v)^*(t^++w_R).
$$
Recall $g^+=t^++w$ in $Z_{R/2}(p)$. Hence  $H_v^*g^+ $ solves local gauge for gauged Einstein equation in $ Z_{R/2}(p)$, that is, (i) is proved. The proof of (ii) is given in \cite[Lemma 4.4]{CDLS}. Thus we finish the proof.
\end{proof}

\subsection{$\varepsilon$-regularity} \label{varepsilon-regularity}
In this part, we want to prove some higher order regularity of $g^+$  up to a diffeomorphism (or equivalently, high order regularity of $H_v^*g^+- t^+$). We establish first the uniform bound for the linearized operator $D_1 F$, where
$F$ is the gauged Einstein functional \eqref{gauged Einstein equationbis}, and its inverse. 
\begin{lemm} 
\label{lemma3bisSection3.1} 
Under the assumptions (H1)-(H3),  there exists a positive constant $C$  independent of $p\in \p X$ (and the sequence of the metrics) such that
$$
\| D_1 F(t^+_{p,\bar R_0},t^+_{p,\bar R_0})\|+ \| (D_1 F(t^+_{p,\bar R_0},t^+_{p,\bar R_0}))^{-1}\|\le C, 
$$
where $D_1 F(t^+_{p,\bar R_0},t^+_{p,\bar R_0}): C_{2+\lambda}^{2,\lambda}(\tilde Z_{ \bar R_0}(p))\to C^{0,\lambda}_{2+\lambda}(\tilde  Z_{\bar R_0}(p))$ (or $D_1 F(t^+_{p,\bar R_0},t^+_{p,\bar R_0}): C_{1+\lambda}^{2,\lambda}(\tilde Z_{ \bar R_0}(p))\to C^{0,\lambda}_{1+\lambda}(\tilde  Z_{\bar R_0}(p))$).
Moreover, such estimates hold also for $D_1 F(t^+_{p,\bar R_0},t^+_{p,\bar R_0}): C_{3+\lambda}^{3,\lambda}(\tilde Z_{ \bar R_0}(p))\to C^{1,\lambda}_{3+\lambda}(\tilde  Z_{\bar R_0}(p)).$ (or $D_1 F(t^+_{p,\bar R_0},t^+_{p,\bar R_0}): C_{1+\lambda}^{3,\lambda}(\tilde Z_{ \bar R_0}(p))\to C^{1,\lambda}_{1+\lambda}(\tilde  Z_{\bar R_0}(p))$).
\end{lemm}
\begin{proof} We state the sectional curvature of $t^+_{p,\bar R_0}$ is negative in $\tilde Z_{ \bar R_0}(p)$. It is known (see \cite[Proof of Theorem A]{Lee1}) the $L^2$ kernel of the operator 
$D_1 F(t^+_{p,\bar R_0},t^+_{p,\bar R_0})$ is trivial. Hence by \cite[Theorem C]{Lee1},  $D_1 F(t^+_{p,\bar R_0},t^+_{p,\bar R_0}):C_{2+\lambda}^{2,\lambda}(\tilde Z_{ \bar R_0}(p)))\to C^{0,\lambda}_{2+\lambda}(\tilde Z_{\bar R_0}(p))$ is an isomorphism since $2+\lambda\in (0,d)$. Recall the family of $t$ is compact in the $C^{2,\lambda}$-Cheeger-Gromov topology for all $\lambda\in (0,1)$ (even $C^{3,\lambda}$). By the same arguments in the proof of Lemma \ref{lemma2bisSection3.1}, the desired results follow. The proof in the high order H\"older spaces  is same. We finish the proof.
\end{proof}

Now we can prove the $\varepsilon$-regularity result.

\begin{lemm} 
\label{lemma3Section3.1} 
Under the assumptions (H1)-(H3),  there exists positive constant $C$ and a small positive constant $\varepsilon$ independent of $p\in \p X$ (and the sequence of the metrics) such that if for all $R<\min(R_1/2,1)$ we have 
$$\|H_v^*g^+- t^+\|_{C_{\lambda}^{0,\lambda}(Z_{R}(p))}\le \varepsilon\mbox{ and }  \|H_v^*g^+- t^+\|_{C_{1+\lambda}^{1,\lambda}(Z_{R}(p))}\le C,$$ 
then we have
$$
\|H_v^*g^+- t^+\|_{C_{2+\lambda}^{2,\lambda}(Z_{{R}/2}(p))}\le\frac{C}{R},
$$
moreover, there holds
$$
\|H_v^*g^+- t^+\|_{C_{3+\lambda}^{3,\lambda}(Z_{{R}/4}(p))}\le\frac{C}{R^2}.
$$
\end{lemm}

\begin{proof}  
We consider the following functional
$$E[u]:=F(t^++u,t^+),$$
where $u$ is a symmetric $2$-tensor fields.  By Lemma  \ref{lemma2bisSection3.1},  $u={\tilde g}^+ -t^+:= (H_v)^*g^+ -t^+ $ is a solution of $E[u]=0$ in $Z_{R_1/2}(p)$.  It is a is quasilinear uniformly degenerate equation with its  
linearized operator at $0,$ $DE[0]=\frac{1}{2}(\Delta_L+2(d-1))=:P,$ which is
of course, a geometric elliptic operator. 
Recall  $ u$ solves $E(u)=0$ in $Z_{R}(p)$. On the other hand, a direct calculation leads to (see \cite{G00})
$$
E(0)=Ric[t^+]+(d-1)t^+\in C^{1,\lambda}_{(1+\lambda)}(\overline{Z_{R}(p)})\subset C^{0,\lambda}_{(\lambda)}(\overline{Z_{R}(p)}).
$$
Thus, by Lemma \ref{Lee1bis}, we deduce
$$
\|E(0)\|_{C_{2+\lambda}^{0,\lambda}(Z_{R}(p))}\le C.
$$
Here the bound $C$ is independent of $p$ and the sequence.  Let $G[u]=E[u]-E[0]-DE[0]u$ be the quadratic polynomials and higher degree in $u$. Hence we can estimate for small $u$
$$
\|G[u]\|_{C^{0,\lambda}_{2+\lambda}(Z_{R}(p))}\le C\|G[u]\|_{C^{0,\lambda}_{2+2\lambda}(Z_{R}(p))}\leq C( \|u\|_{C^{2,\lambda}_{2+\lambda}(Z_{R}(p))} \|u\|_{C^{0,\lambda}_{\lambda}(Z_{R}(p))}+ \|u\|^2_{C^{1,\lambda}_{1+\lambda}(Z_{R}(p))}),
$$
where $C$ is independent of $p$ and the sequence of the metrics. As $t$ is in in the $C^4$ space, we can choose  $\varphi_R(x)=\varphi(d_{t}(x,p)/R)$ be the  $C^3$ cut-off function in $ Z_{\bar R_0}(p)$ such that
 \begin{align*}
\mbox{supp}\,\varphi_R\in Z_{R}(p),\ \varphi_R\equiv 1 \ on\ Z_{\frac{R}{2}}(p),\\
\|\nabla^k\varphi_R(x) \|\leq C_0R^{-k},  \forall 0\le k\le 3. 
\end{align*}
We have 
$$
\varphi_R G[u] =\varphi_R( E[u]-E[0]-DE[0]u)=- \varphi_RE[0]-P( \varphi_Ru)+ [\varphi_R, P]u.
$$
We note 
$$
[\varphi_R, P]u=-\nabla^*  u\nabla  \varphi_R-u \nabla^*  \nabla  \varphi_R, 
$$
so that we have
\begin{align*}
\|[\varphi_R, P]u\|_{C^{0,\lambda}_{2+\lambda}(Z_{R}(p))}\le C(\frac{\|u\|_{C^{1,\lambda}_{1+\lambda}(Z_{R}(p))}}{R}+ \frac{\|u\|_{C^{0,\lambda}_{1+\lambda}(Z_{R}(p))}}{R}).
\end{align*}
For this purpose,  let $\Phi_l$ be any M\"obius chart around some point $p_l\in  Z_{R}(p)$. We write
\begin{align*}
\Phi_l^*([\varphi_R, P]u)=& -\nabla^* \Phi_l^*u\nabla \Phi_l^*\varphi_R-\Phi_l^*u \nabla^* \nabla \Phi_l^*\varphi_R,
\end{align*}
where the connection $\nabla$ is related to the metric $ \Phi_l^*t^+$. Thus,
\begin{align*}
\|\Phi_l^*([\varphi_R, P]u)\|_{0,\lambda;B_1}& \le C( \|\nabla \Phi_l^*u\|_{0,\lambda;B_1}\|\nabla \Phi_l^*\varphi_R\|_{0,\lambda;B_1}+\| \Phi_l^*u\|_{0,\lambda;B_1}\|\nabla^2 \Phi_l^*\varphi_R\|_{0,\lambda;B_1})\\
&\le  C( \| \Phi_l^*u\|_{1,\lambda;B_1}\|\nabla \Phi_l^*\varphi_R\|_{0,\lambda;B_1}+\| \Phi_l^*u\|_{0,\lambda;B_1}\|\nabla^2 \Phi_l^*\varphi_R\|_{0,\lambda;B_1})\\
&\le  C\rho(p_l)^{2+\lambda}(\|u\|_{C^{1,\lambda}_{1,\lambda}}/R+ \|u\|_{C^{0,\lambda}_{1,\lambda}}/R),
\end{align*}
where $C$ is some positive constant independent of $p$ and the sequence of the metrics. Thus, the desired estimate follows. Now we estimate
\begin{align*}
\|\varphi_R E(0)\|_{C_{2+\lambda}^{0,\lambda}(Z_{R}(p))}
\le C\| \varphi_R \|_{C^{0,\lambda}(Z_{R}(p))} \|E(0) \|_{C_{2+\lambda}^{0,\lambda}(Z_{R}(p))}\le C. 
\end{align*}
Similarly
\begin{align*}
\|\varphi_R G[u]\|_{C^{0,\lambda}_{2+\lambda}(Z_{R}(p))}&\le C\|\varphi_R G[u]\|_{C^{0,\lambda}_{2+2\lambda}(Z_{R}(p))}\\
&\leq C( \|\varphi_R u\|_{C^{2,\lambda}_{2+\lambda}(Z_{R}(p))} \|u\|_{C^{0,\lambda}_{\lambda}(Z_{R}(p))}+ \| \varphi_R \|_{C^{0,\lambda}(Z_{R}(p))}\|u\|^2_{C^{1,\lambda}_{1+\lambda}(Z_{R}(p))})\\
&\leq C( \|\varphi_R u\|_{C^{2,\lambda}_{2+\lambda}(Z_{R}(p))} \|u\|_{C^{0,\lambda}_{\lambda}(Z_{R}(p))}+\|u\|^2_{C^{1,\lambda}_{1+\lambda}(Z_{R}(p))}).
\end{align*}
Gathering the above estimates, we infer 
\begin{align*}
&\|-\varphi_R G[u] - \varphi_RE[0]+ [\varphi_R, P]u\|_{C^{0,\lambda}_{2+\lambda}(Z_{R}(p))}\\
\le&C(\|\varphi_R u\|_{C^{2,\lambda}_{2+\lambda}(Z_{R}(p))} \|u\|_{C^{0,\lambda}_{\lambda}(Z_{R}(p))}+(1+\|u\|^2_{C^{1,\lambda}_{1+\lambda}(Z_{R}(p))})/R),
\end{align*}
provided $R<1$.  Now we write
$$
P( \varphi_Ru)=-\varphi_R G[u] - \varphi_RE[0]+ [\varphi_R, P]u.
$$
Given a section $w$ on $Z_{R}(p)$, let us denote  $\tilde w:=\Psi^*_{p,\bar R_0}w$ and $\tilde P, \tilde E,\tilde G$ the pull back by $\Psi_{p,\bar R_0}$ of $P,E,G$. It is clear
$$
\|w\|_{k,\lambda;\delta}=(\bar R_0)^\delta \|\tilde w\|_{k,\lambda;\delta}.
$$
Hence
\begin{align*}
&\|-\tilde{\varphi_R G[u]} - \tilde {\varphi_RE[0]}+ \tilde{[\varphi_R, P]u}\|_{C^{0,\lambda}_{2+\lambda}(\tilde Z_{\bar R_0}(p))}\\
\le&C(\|\tilde {\varphi_R u}\|_{C^{2,\lambda}_{2+\lambda}(\tilde Z_{\bar R_0}(p))} \|u\|_{C^{0,\lambda}_{\lambda}(Z_{R}(p))}+(1+\|u\|^2_{C^{1,\lambda}_{1+\lambda}(Z_{R}(p))})(\bar R_0)^{2+\lambda}/R).
\end{align*}
We know $ \tilde{\varphi_Ru}\in C^{1,\lambda}_{1+\lambda}(\tilde Z_{\bar R_0}(p))$ and $\tilde P( \tilde {\varphi_Ru})\in C^{0,\lambda}_{2+\lambda}(\tilde Z_{\bar R_0}(p))\subset C^{0,\lambda}_{1+\lambda}(\tilde Z_{\bar R_0}(p))$ which implies by \cite[Lemma 4.8]{Lee1} $ \tilde{\varphi_Ru}\in C^{2,\lambda}_{1+\lambda}(\tilde Z_{\bar R_0}(p))$.  Therefore, applying Lemma \ref{lemma3bisSection3.1}, we can write
$$
 \tilde {\varphi_Ru}=\tilde P^{-1} (-\tilde{\varphi_R G[u] }- \tilde{\varphi_RE[0]}+\tilde{ [\varphi_R, P]u})\in C^{2,\lambda}_{2+\lambda}(\tilde Z_{\bar R_0}(p))\subset C^{2,\lambda}_{1+\lambda}(\tilde Z_{\bar R_0}(p)).
$$
Again from Lemma \ref{lemma3bisSection3.1}, we can obtain
\begin{align*}
\|  \tilde {\varphi_Ru}\|_{C^{2,\lambda}_{2+\lambda}(\tilde Z_{\bar R_0}(p))}
\le C(\|\tilde {\varphi_R u}\|_{C^{2,\lambda}_{2+\lambda}(\tilde Z_{\bar R_0}(p))} \|u\|_{C^{0,\lambda}_{\lambda}(Z_{R}(p))}+(1+\|u\|^2_{C^{1,\lambda}_{1+\lambda}(Z_{R}(p))})(\bar R_0)^{2+\lambda}/R),
\end{align*}
so that
\begin{align*}
(1-C\|u\|_{C^{0,\lambda}_{\lambda}(Z_{R}(p))}) \| \varphi_Ru\|_{C^{2,\lambda}_{2+\lambda}( Z_{R}(p))}
\le C(1+\|u\|^2_{C^{1,\lambda}_{1+\lambda}(Z_{R}(p))})/R.
\end{align*}
Now, we take $1-C\|u\|_{C^{0,\lambda}_{1+\lambda}(Z_{R}(p))}\le 1/2$, and the desired result follows.\\
For the high order regularity, we state first that 
$$
E(0)=Ric[t^+]+(d-1)t^+\in C^{1,\lambda}_{(1+\lambda)}(\overline{Z_{R}(p)})
$$
so that by Lemma \ref{Lee1bis}
$$
 \|E(0)\|_{C_{3+\lambda}^{1,\lambda}(Z_{R}(p))}\le C.
$$
The proof for the rest is similar. Therefore, we finish the proof.
\end{proof}

Now, we could establish the high order regularity of $g^+$ in a neighborhood of conformal infinity up to a diffeomorphism (or equivalently, high order regularity of $H_v^*g^+- t^+$). Namely, we have
\begin{lemm} 
\label{lemma3bis2Section3.1} 
Under the assumptions (H1)-(H3),  there exists positive constant $C$ and small positive constant $\bar R_1<\min(R_1,1)$ independent of $p\in \p X$ (and the sequence of the metrics) such that
$$
\|H_v^*g^+- t^+\|_{C_{2+\lambda}^{2,\lambda}(Z_{\bar R_1}(p))}\le\frac{C}{\bar R_1}.
$$
Moreover,  we have
$$
\|H_v^*g^+- t^+\|_{C_{3+\lambda}^{3,\lambda}(Z_{\bar R_1}(p))}\le\frac{C}{\bar R_1^{2}}.
$$
\end{lemm}

\begin{proof}  We claim $\|H_v^*g^+- t^+\|_{C_{1+\lambda}^{1,\lambda}(Z_{R}(p))}\le CR^{\tilde\lambda-\lambda}$ with $0<\lambda<\tilde\lambda<1$. We write $w=g^+-t^+$ so that
$$
H_v^*g^+- t^+= H_v^*t^+- t^++ H_v^*w.
$$
Thanks of Lemmas \ref{lemma2bisSection3.1} and \ref{lemma2Section3.1} , we estimate
$$
\|H_v^*t^+- t^+\|_{C^{1,\lambda}_{1+\lambda}(Z_R(p))}\le CR^{\tilde\lambda- \lambda}.
$$
On the other hand,   for sufficiently small $v\in {C}_{1+\lambda}^{2,\lambda}(Z_{\bar R_0}(p); TX)$, $ H_v:  Z_{R_1}(p)\to Z_{2R_1}(p)$ is a diffeomorphism. 
As same as in \cite[Lemmas 4.2 and 4.4]{CDLS}, set $A(x)= H(x)-x$ in M\"obius chart around some point $p_l\in Z_{\bar R}(p)$. Therefore, we have
\begin{align*}
 \|A(x)\|_{2,\lambda;\bar{B}_2}&\leq C\|\Phi_i^*v\|_{2,\lambda;\bar{B}_2}
                              \leq C \rho(p_i)^{1+\lambda} \|v\|_{C^{2,\lambda}_{1+\lambda}( Z_{\bar R_0}(p))}.
\end{align*}
Here $C$ is some positive constant independent of  $p\in \p X$ and the sequence of metrics. Therefore, we obtain
\begin{align*}
\Phi_l^*((H_v)^*{w})=&((\Phi_l^*w)_{jk}(H(x)) dx^j\otimes dx^k+2(\Phi_l^*w)_{jk}(H(x))\frac{\partial A^k}{\partial x^q}dx^j\otimes dx^q\\
                &+(\Phi_l^*w)_{jk}(H(x))\frac{\partial A^j}{\partial x^m}\frac{\partial A^k}{\partial x^q}dx^m\otimes dx^q.
\end{align*}
In view of Lemma \ref{lemma1Section3.1},  we can estimate
\begin{align*}
\| \Phi_l^*((H_v)^*{w})\|_{1,\lambda;\bar{B}_1}\le  C \|\Phi_l^*({w}) \|_{1,\lambda;\bar{B}_2}\le C \rho(p_l)^{1+\tilde\lambda}\le C \rho(p_l)^{1+\lambda}R^{\tilde\lambda-\lambda}.
\end{align*}
As a consequence, we infer
\begin{align*}
\|H_v^*w\|_{C_{1+\lambda}^{1,\lambda}(Z_{R}(p))}\le CR^{\tilde\lambda-\lambda}.
 \end{align*}
Therefore, we prove the claim. Now, we choose small $\bar R_1$ such that $C(4\bar R_1)^{\tilde\lambda-\lambda}<\varepsilon$.  The desired result yields. We finish the proof.
\end{proof}

\subsection{Regularity of  $g^*$} \label{Regularity of  $g^*$}
In this part, we want to get the regularity of  adapted metric $g^*$. For this purpose, our key observation is to obtain first the regularity result for the Cotton tensor (or the Bach tensor).
\begin{lemm} 
\label{lemma4bisSection3.1} 
Under the assumptions (H1)-(H3),  there exists some positive constant $C>0$ independent of the sequence of metrics and of $p\in \p X$ (depending on $\lambda$ and $\bar R_1$) such that 
 there holds   in $ Z_{{\bar R_1}}(p)$
$$
\|W[g^*]\|_{C^{1,\lambda}(Z_{{\bar R_1}}(p))}\le C,
$$
and 
$$
\|\mathcal{C}[g^*]\|_{C^{0,\lambda}(Z_{{\bar R_1}}(p))}\le C.
$$
\end{lemm}
\begin{proof}
We write $g^*=\rho^2 g^+= ( H_v^{-1})^*(\frac{(\rho\circ H_v)^2}{\rho^2}\rho^2H_v^*g^+)$ and $g_1=\frac{(\rho\circ H_v)^2}{\rho^2}\rho^2H_v^*g^+=H_v^*g^*$. It follows from Lemmas \ref{lemma1Section3.1} to \ref{lemma3bis2Section3.1} that
$$
\|H_v^*g^+\|_{C_{3+\lambda}^{3,\lambda}( Z_{\bar R_1}(p))}\le C,
$$
so that by Lemma \ref{Lee1bis}, the compactified metric verifies
$$
\|\rho^2 H_v^*g^+\|_{C^{3,\lambda}(\overline{ Z_{\bar R_1}(p)})}\le C.
$$
Thus  $ \rho^2 H_v^*g^+$ has bounded curvature  in the $C^{1,\lambda}$ space or all $\lambda\in (0,1)$. We recall the Weyl tensor is a local conformal invariant when the dimension $d\ge 4$, that is, the Weyl tensor as a $(3,1)$ tensor, we have
$$
W[\rho^2 H_v^*g^+]= W[ H_v^*g^+]=H_v^*W[g^+]=H_v^*W[g^*].
$$
Thus, $((H_v)^{-1})^*W[\rho^2 H_v^*g^+] =W[g^*]$. Recall   $H_v$ is a $C^{2,\lambda}$ diffeomorphism so that 
$$
\|W[g^*]\|_{C^{1,\lambda}(\bar Z_{\bar R_1}(p))}\le C .
$$
It is known that 
$$
\mathcal{C}[g^*]_{ijk}=\frac{1}{d-3}{W[g^*]_{jkil,}}^{l}.
$$
Therefore, we infer 
$$
\|\mathcal{C}[g^*]\|_{C^{0,\lambda}(Z_{\bar R_1}(p))}\le C.
$$
Hence, we prove the desired result.
\end{proof}

\begin{lemm} 
\label{lemma4Section3.1} 
Under the assumptions (H1)-(H3),  there exists some positive constant $C>0$ independent of the sequence of metrics and of $p\in \p X$ (depending on $\lambda$ and $\bar R_1$) such that 
 there holds  in $ Z_{{\bar R_1/2}}(p)$
$$
\|Rm_{g^*}\|_{C^{1,\lambda}}\le C.
$$
\end{lemm}
\begin{proof}
It is known the Bach tensor can be written 
$$
B_{ij}=\nabla^k \mathcal{C}_{ijk}+A^{kl}W_{ikjl}.
$$
Using the equation (\ref{Bach-eq-1}), we can write by Lemma \ref{lemma4bisSection3.1} 
$$
\triangle R_{ij}=\p_k f_k+ g,
$$
where $f_k\in C^{0,\lambda}$   and $g\in L^\infty$. Assume that $h$ is in the $C^4$ space on $\p X$ so that $Ric|_{\p X}$ is in the $C^2$ space on the boundary. By the classical regularity theory, for example  \cite[theorem 8.33]{GT}, there holds
$$
\|Ric[g^*]\|_{C^{1,\lambda}(\bar Z_{\bar R_1/2}(p))}\le C. 
$$
Finally, by the decomposition of Riemann curvature tensor, we prove the desired result.
\end{proof}

\begin{RK}
 We expect the higher order $C^{d-2, \gamma}$ regularity result for all dimensions $d\ge 4$
of $g^*$  provided the representative metric $h$ at the conformal infinity also 
satisfies some sufficient higher order regularity. For example, we have compactness results for $g_i^{*}$ in the $C^{k,\gamma'}$ norm with $2\le k\le d-2$ when $\{h_i\}$ is a compact family in the $C^{k+1,\gamma}$ space  with $1\ge\gamma>\gamma'$.   In fact, we could expect to construct more regular approximated metrics $t$ by the use of the same representative metric $h$ on the conformal infinity, which is different than (\ref{tequation}).
\end{RK}

\section{Proof of Theorems \ref{maintheorem3bis} and \ref{maintheorem3}}

\label{Sect:compactness}

We are now ready to establish the two compactness theorems for our adapted metrics
on conformally compact Einstein manifolds stated in the introduction of any $d$-dimensions.  As we have indicated before, the strategy of the proof of the two theorems follows closely from the corresponding results in dimension $d=4$ \cite{CG, CGQ} with the main difference in the step in the gaining of the regularity of the adapted metrics when the dimension $d$ is higher. Here when $d$ is even,  we will use the existence of the obstruction tensor to gain the regularity, and we will carry out the proof of 
Theorem \ref{maintheorem3bis} in more details below. When the dimension $d$ may not be even, we will use the gauged Einstein equation method which we have described in detail in Section \ref{Sect:bdy2}, to derive the $\epsilon$ 
regularity of the curvature and hence the  higher order regularity of the adapted metrics. Once this step is accomplished, the rest of the proof of Theorem \ref{maintheorem3} is essentially the same as the proof of Theorem \ref{maintheorem3bis}. Thus we
will state the result and omit the details of the proof. 

 
 \subsection{Proof of Theorem \ref{maintheorem3bis}}
To begin the proof, we will first establish some upper bounds of the curvature and its derivatives of the adapted metrics.

\begin{lemm} \label{Lem:curv-estimate} 
Suppose that $\{(X_i^d, g^+_i)\}$ is a sequence of conformally compact Einstein even $d$-dimensional manifolds satisfying
the assumptions  in Theorem \ref{maintheorem3bis}. 
Then there exists a positive constant $K_0$ such that, for the adapted metrics $\{(X_i^d, g_i^*)\}$ 
associated with a compact family of boundary metrics $h_i$ --a representative of the conformal infinity $(\p X_i^d, [h_i])$, we have
\begin{equation}\label{Eq:curvature-bound}
\max_{X_i^d} \sup_{k=(k_1,\cdots,k_l), \;|k|:=l\le d-4}|\nabla^k Rm_{g_i^*}|^{\frac2{|k|+2}} \le K_0
\end{equation}
for all $i$.
\end{lemm}

We remark that the constant $K_0$ in the statement of the lemma above depends on the smallness of the constant  $\delta_0$ which appears among  the assumptions of Theorem  \ref{maintheorem3bis}.
 
Suppose otherwise that there is a subsequence $\{(X_i^d, g_i^+)\}$ satisfying 
$$
K_i  = \max_{X_i}\sup_{k=(k_1,\cdots,k_l), \;|k|:=l\le d-4}|\nabla^k Rm_{g_i^*}|^{\frac2{|k|+2}}\to \infty,
$$
and either
\beq
\label{C1}
\int_{X^d} (|W_{g_i^+}|^{d/2}dvol)[g^+_i] \to 0,
\eeq
or
\beq
\label{C2}
Y(\partial X, [h_i]) \to Y(\mathbb{S}^{d-1}, [g_{\mathbb{S}}]).
\eeq
Let 
$$
K_i = K_i(p_i) =  \max_{k=(k_1,\cdots,k_l), \;|k|:=l\le d-4}|\nabla^k Rm_{g_i^*}|^{\frac2{|k|+2}}(p_i)
$$
for some $p_i\in \overline{X_i}$. Then we consider the rescaling
$$
(X^d_i, \bar g_i = K_ig_i^*, p_i).
$$
In view of Lemmas \ref{bdy-injrad} and \ref{int-injrad}, we have the uniform lower bound of the intrinsic  injectivity radius $ i_{\text{int}} (X, \bar g_i)$ and of the boundary  injectivity radius $ i_{\p} (X, \bar g_i)$. Together with the assumption on the conformal infinity, we know the intrinsic  injectivity radius $ i(\p X, \hat{{\bar g}}_i:={{\bar g}}_i|_M)$ on the boundary  is also uniformly bounded from below.  Thus, for given $M>1$, the harmonic radius $r^{1,\gamma}(M)$ (see \cite[Section 2.4]{CGQ}) is uniformly bounded from below for the family of metrics $\bar g_i$.  Hence, the assumptions (1) to (3) in Lemma \ref{bdy1} are satisfied for such metrics $\bar g_i $. Applying Lemmas \ref{bdy1} and \ref{AKKLT}, we have the compactness result in the $C^{d-2,\gamma'}$-Cheeger-Gromov topology with base points for the metrics $\bar g_i$ with $\gamma'<\gamma$, provided that the conformal infinity is bounded in the $C^{d-2,\gamma}$ norm. The proof is divided into two parts: no boundary blow-up (Lemma \ref{No boundary Blow-up}), and no  interior blow-up (Lemma \ref{No interior blow-up}).

\begin{lemm} \label{No boundary Blow-up} 
Under the  assumptions  in Theorem \ref{maintheorem3bis}, there is no blow-up near the boundary.
\end{lemm}
\begin{proof}
We argue by contradiction. Let us first consider the cases where 
$$
\text{dist}_{\bar g_i} (p_i, \p X_i) < \infty.
$$
For the pointed manifolds $(X_i, \bar g_i, p_i)$ with boundary, in the light of all the preparations in the previous sections, we have the 
Cheeger-Gromov convergence
$$
(X^d_i, \bar g_i, p_i) \to (X^d_\infty, g_\infty, p_\infty)
$$
in 
the $C^{d-2, \gamma'}$-Cheeger-Gromov topology (up to a subsequence if necessary), where the limit space is a 
complete manifold  with Q-flat and vanishing obstruction tensor in the distribution sense, and with a totally geodesic boundary $\p X_\infty$.  We have 
$$
\max_{k=(k_1,\cdots,k_l), \;|k|:=l\le d-4}|\nabla^k Rm_{g_\infty}|^{\frac2{|k|+2}}(p_\infty)= 1.
$$ 
We also observe that the boundary 
$(\p X_\infty, h_\infty)$ is simply the Euclidean space $\mathbb{R}^{d-1}$ due to the assumption that the boundary metrics $\{h_i\}$ form a compact family.



To finish the proof, it is sufficient  to show that the limit space $(X^d_\infty, g_\infty, p_\infty)$ is a locally Euclidean space. 
For the convenience of readers, we very briefly sketch the proof from \cite{CG,CGQ}. One first needs to show that 
$\bar\rho_i \to \rho_\infty$ where $\rho_\infty$ satisfies
\begin{itemize}
\item $g^+_\infty = \rho_\infty^{-2}g_\infty$ is a (partially) conformally compact Einstein metric on $X^d_\infty$ whose conformal
infinity is the Euclidean space $\mathbb{R}^{d-1}$;
\item $v_\infty=\rho_\infty^{\frac{d-4}{2}}$ solves $- \Delta_{g^+_\infty} v_\infty -\frac{(d-1)^2-9}{4}v_\infty= 0$.
\end{itemize}
Then, by Condition (\ref{C1}), one shows that $g^+_\infty$ is Weyl free and is locally hyperbolic space metric. \\
Now we assume Condition (\ref{C2}). We choose $q_i\in X$ such that $d(q_i,\p X)\ge 1$ and $d(p_i,q_i)$ is bounded so that $(X^d_i, g_i^+, q_i) \to (X^d_\infty, g_\infty^+, q_\infty)$ in the  
$C^{d-2, \gamma'}$-Cheeger-Gromov topology with based points.  It follows from Lemma \ref{Lem:LQS} that for any $r>0$
$$
1= \frac {\text{vol}_{g^+_\infty}(B(q_\infty, r))}{\text{vol}_{g_{\mathbb{H}^d}}(B(r))}, 
$$
so that $g^+_\infty$ is locally hyperbolic space metric by the Bishop-Gromov's volume comparison Theorem.

We now apply a proof similar to that of  \cite[Proposition 4.8]{CG} when $d=4$, with the modification to the case when dimension $d>4$. We work with the limit metric. For simplicity, we omit the index $\infty$. We denote $\tilde g^+$  the standard hyperbolic space with the upper half space model. As $\tilde g^+= g^+$ in a neighborhood of the boundary $\{x_1=0\}$, we can extend this local isometry to a covering map  $\pi: \tilde g^+\to g^+$. We write 
 $$
 g_1= x_1^2 \tilde g^+ \mbox{ and } g_2=  \rho^2g^+,
 $$
where $g_1$ is the standard Euclidean metric and $g_2$ the limiting adapted metric. With the help of the covering map $\pi$, we have $\pi^*g_2=\tilde  \rho^2g^+$ where $\tilde\rho= \rho\circ \pi$. 
We have 
\beq
\label{eqEigen}
-\triangle_{\tilde g^+}\tilde  \rho^{\frac{d-4}{2}}-\frac{(d-1)^2-9}{4}\tilde  \rho^{\frac{d-4}{2}}= 0.
\eeq
Also, it is evident
$$- \Delta_{\tilde g^+} x_1^{\frac{d-4}{2}} -\frac{(d-1)^2-9}{4}x_1^{\frac{d-4}{2}}= 0.
$$
Recall that $x_1$ is the geodesic defining function with respect to the flat boundary metric. We write $\pi^*g_2=\tilde \rho^2 g^+=(\frac{\tilde \rho}{x_1})^2g_1=: u^{\frac{4}{d-4}}g_1$ where $u=(\frac{\tilde  \rho}{x_1})^{\frac{d-4}{2}}$. The semi-compactified metric $g_2$ (or $\pi^*g_2$) has flat $Q_4$  and  the boundary metric of $g_2$ is the $(d-1)$-dimensional  Euclidean space and totally geodesic. Thus $u$ satisfies
the following conditions:
\beq
\label {eq4.5bis}
\left\{
\begin{array}{lll}
\triangle^2 u =0&\mbox{ in } \R^d_+, \\
-\frac{\triangle u}u-\frac{2}{d-4}\frac{|\nabla u|^2}{u^2} \ge 0&\mbox{ in } \R^d_+, \\
u=1&\mbox{ on } \p\R^d_+, 
\\
\nabla u =\triangle u=0&\mbox{ on } \p\R^d_+.
\end{array}
\right.
\eeq
The first equation comes from the flat $Q_4$ curvature and second one from the non-negative scalar curvature. As $g_2$ on the boundary is Euclidean, $u$ on the boundary is equal to constant $1$.  On the other hand, we know both $g_1$ and $g_2$ have the totally geodesic boundary. Hence on the boundary, $\p_1 u=0$ so that  $\nabla u=0$. On the other hand, it follows from Lemma 2.3 the restriction of the scalar curvature vanishes on the boundary so that $-\triangle u-\frac{2}{d-4}|\nabla u|^2 = 0$. This yields $\triangle u=0$ on the boundary. On the other hand, we know that $-\triangle  u\ge 0$ in  $\R^d_+$.\\
Using a result due to H.P. Boas and R.P. Boas \cite{BB}, there exists some $a\ge0$ such that
\beq
\label{eqEigenbis}
-\triangle  u = ax_1.
\eeq
We denote $w:={\tilde  \rho}^{\frac{d-4}2}$. Then, equation (\ref{eqEigen}) is equivalent to the following one
$$
\triangle w+\frac{2-d}{x_1}\p_1 w=-\frac{(d-1)^2-9}{4x_1^2}w,
$$
so that
$$
\triangle  u=\frac{\triangle w}{x_1^{\frac{d-4}{2}}}-\frac{d-4}{x_1^{\frac{d-2}{2}}}\p_1 w+\frac{(d-2)(d-4)}{4x_1^{\frac{d}{2}}} w=\frac{2}{x_1^{\frac{d-2}{2}}}\p_1 w+\frac{4-d}{x_1^{\frac{d}{2}}}w.
$$
Together with (\ref{eqEigenbis}), we infer
$$
\p_1 w+\frac{4-d}{2x_1}w=  -\frac{a}{2}x_1^{\frac d2}.
$$
Therefore, for fixed $(x_1^0,x_2^0,\cdots, x_d^0)$ with $x_1^0>0$, we have for $t>0$
$$
t^{\frac{4-d}2} w(t,x_2^0,\cdots, x_d^0)-(x_1^0)^{\frac{4-d}2} w(x_1^0,x_2^0,\cdots, x_d^0)=
-\frac{a}{6}(t^3-(x_1^0)^3).
$$
Taking $t\to +\infty$, we infer
$$
\begin{array}{l}
-(x_1^0)^{\frac{4-d}2}  w(x_1^0,x_2^0,\cdots, x_d^0)\le \lim t^{\frac{4-d}2} w(t,x_2^0,\cdots, x_d^0)-(x_1^0)^{\frac{4-d}2} w(x_1^0,x_2^0,\cdots, x_d^0) \\
=\lim -\frac{a}{6}(t^3-(x_1^0)^3)=-\infty,
\end{array}
$$
provided $a>0$. This gives also a contradiction when $a>0$. Hence $a=0$. Finally, $-\frac{\triangle u}u-\frac{2}{d-4}\frac{|\nabla u|^2}{u^2} \ge 0$ implies $\nabla u\equiv 0$, that is, $g_2$ is flat. This contradiction yields that there is no boundary blow-up.
\end{proof}

\begin{lemm} \label{No interior blow-up} 
Under the  assumptions  in Theorem \ref{maintheorem3bis}, there is no interior  blow-up.
\end{lemm}
\begin{proof} We consider the remaining case when
$$
\text{dist}_{\bar g_i}(p_i, \p X_i) \to \infty
$$
(at least for some subsequence). Notice that, 
$$
K_i = \max_{X_i}  \max_{k=(k_1,\cdots,k_l), \;|k|:=l\le d-4}|\nabla^k Rm_{g_i^*}|^{\frac2{|k|+2}} = \max_{k=(k_1,\cdots,k_l), \;|k|:=l\le d-4}|\nabla^k Rm_{g_i^*}|^{\frac2{|k|+2}}(p_i)
$$
for some $p_i\in X$ in the interior. Proceeding as the above boundary cases, one has the Cheeger-Gromov convergence
$$
(X^d_i, \bar g_i, p_i) \to (X^d_\infty, g_\infty, p_\infty)
$$
in 
 the $C^{d-2, \gamma'}$-Cheeger-Gromov topology. The proof in these cases follows from \cite{CG}.  We again very briefly sketch the proof that is 
more or less from \cite{CG}. One first derives from \eqref{relation3} that
$$
R_{\bar g_i} = 2(d-1)\bar\rho_i^{-2} (1 - |d\bar\rho_i|^2_{\bar g_i}).
$$
We also have 
\begin{itemize}
\item $\bar\rho_i (x) \ge C \text{dist}_{\bar g_i} (x, \p X_i)$. (cf. Step 2 in the proof of \cite[Lemma 4.9]{CG}).
\end{itemize}
Consequently, 
\begin{itemize}
\item $R_\infty =0$, and 
\item $g_\infty$ is Ricci-flat from being $Q$-flat and scalar flat in light of the $Q$-curvature equation \eqref{Q-eq}. (cf. Step 3 of
the proof of \cite[Lemma 4.9]{CG}).
\end{itemize}
Thus, $(X_\infty, g_\infty)$ is a complete Ricci-flat $d$-dimensional manifold with no boundary. 
As same arguments as in the previous part, we have  $(X_\infty, g_\infty)$ is locally conformally flat,  so that $(X_\infty, g_\infty)$ is flat because of the decomposition of the curvature tensor. Therefore, we obtain the desired contradiction.  
 For more details see \cite[Section 4.3]{CG}.
\end{proof}
\begin{proof}[ Proof of Lemma \ref{Lem:curv-estimate}] It is a direct consequence of Lemmas \ref{No boundary Blow-up}  and \ref{No interior blow-up}.
\end{proof}

We now begin the proof of Theorem \ref{maintheorem3bis}. For this purpose, we first establish the diameter bound.

\begin{lemm} 
\label{diameter}
Under the assumptions in Theorem \ref{maintheorem3bis}, the diameters of the adapted metrics
$g_i^*$ are uniformly bounded. 
\end{lemm}

\begin{proof} We use the similar strategy as in  \cite[Section 4: {\it The proof of Lemma 4.2}]{CGQ}.   We indicate the difference.\\
Thanks to (\ref{relation3}) and (\ref{relation4}),  we infer
$$
-\triangle \sqrt {\rho_i}=\frac {(d+2)R_i\rho_i^{1/2}}{8(d-1)}+\frac{|\nabla \rho_i|^2}{4\rho_i^{3/2}}=\frac{(d+2)(1-|\nabla \rho_i|^2)}{4\rho_i^{3/2}}+\frac{|\nabla \rho_i|^2}{4\rho_i^{3/2}}.
$$
Thus, there exists some constant $C_2>0$ independent of $i$ such that 
\beq
\label{eqlem4.2}
\ds\int_{\{x, d_{ g_i^*}(x,\p X)\ge 1\}} \rho^{-3/2}_i (x)\le C_2.
\eeq
The rest of the proof is almost as same as in the case of dimension $4$. 

\end{proof}

\begin{proof}[ Proof of  Theorem \ref{maintheorem3bis}]
Thanks to Lemmas \ref{Lem:curv-estimate} and
\ref{diameter}, we can use the Cheeger-Gromov compactness result to prove Theorem \ref{maintheorem3bis} (see \cite[Section 5]{CG}). Hence, we finish the proof.
\end{proof}

\subsection{Proof of Theorem \ref{maintheorem3}}

The proof when the dimension $d$ may not be even follows the same outline as the cases when $d$ is even once one manages to gain on the regularity of the curvature tensors. We summarize it in the following lemma.
\begin{lemm} \label{Lem:curv-estimatebis} 
Suppose that $\{(X_i^d, g^+_i)\}$ is a sequence of conformally compact Einstein $d$-dimensional manifolds with all $d\ge 4$ satisfying
the assumptions  in Theorem \ref{maintheorem3}. 
Then there exists a positive constant constant $K_0$  such  that, for the adapted metrics $\{(X_i^d, g_i^*)\}$ 
associated with a compact family of boundary metrics $h_i$ --a representative of the conformal infinity $(\p X_i^d, [h_i])$, we have
\begin{equation}\label{Eq:curvature-boundbis}
\max_{X_i^d} | Rm_{g_i^*}| \le K_0
\end{equation}
for all $i$.
\end{lemm}

We remark again that the constant $K_0$ in the statement of the lemma above depends on the smallness of the constant  $\delta_0$ which appears among the assumptions of Theorem  \ref{maintheorem3}.

First of all, one can establish properties in the hypotheses (H1) (H2) and (H3) in the statements of Lemmas \ref{lemmaSection3.1}  to \ref{lemma4Section3.1}  in Section \ref{Sect:bdy2}  for our normalized adapted metrics $\bar g_i$ by the same procedures of proof as Lemma \ref{Lem:curv-estimate} in this section. We can then replace the role of Lemma \ref{bdy1} in the proof of Theorem \ref{maintheorem3bis} by Lemma  \ref{lemma4Section3.1} to establish Lemma \ref{Lem:curv-estimatebis}, hence the proof of Theorem \ref{maintheorem3}.

\section{Uniqueness of Graham-Lee solutions in high dimension and a gap phenomenon} \label{Sect:uniqueness}

In this section we will derive the global uniqueness result Theorem \ref{uniqueness} and also indicate a gap phenomenon in the Corollary \ref{gap} below, both will be derived as consequences of our compactness Theorem \ref{maintheorem3}.

\begin{proof}[Proof of Theorem \ref{uniqueness}] 
The proof is almost the same as in the case when the dimension of the manifold is $4$ \cite[Section 5]{CGQ}. 
We will sketch the outline of the proof below. 

We will establish the result by a contradiction argument. Assume otherwise there is a sequence of conformal $(d-1)$-dimensional spheres $(\mathbb{S}^{d-1}, [h_i])$ 
that converges to the round sphere such that, for each $i$, there exists two non-isometric conformally compact Einstein metrics $g^+_i$
and $\tilde g^+_i$. 
\\

Up to a subsequence, both $g^+_i$ and $\tilde g^+_i$ converge to the hyperbolic space in the $C^{3,\gamma'}$-Cheeger-Gromov sense (in particular in the $C^{2, \gamma'}$-Cheeger-Gromov sense) due to Theorem \ref{maintheorem3} and the uniqueness result when the conformal infinity is the standard sphere \cite{Q,LQS}.\\

The main facts are the following:
\begin{itemize}
\item There exists a  diffeomorphism $\varphi_i$  of class $C^{2,\gamma}$ for any $\gamma\in (0,1)$ (equal to the identity on the boundary) (see Lemma \ref{lemma2Section3.1}), such that 
$$
 F(\varphi_i^* \tilde g^+_i, g_{i}^+)=0
$$
Moreover $\|\varphi_i(x)-x\|_{C^{2,\gamma}}\to 0$ and $\|\varphi_i^* \tilde g^+_i-g_{i}^+\|_{C^{1,\gamma}_{1+\gamma}}\to 0$.
\item Due to the local uniqueness result (see Lemma \ref{lemma3Section3.1}), for large $i$, we have 
$$
g_{i}^+=\varphi_i^*\tilde g^+_i.
$$

\end{itemize}

\end{proof}

As a direction consequence  of Theorem \ref{maintheorem3}, we are able to prove some gap phenomenon. Given some large positive number $\Lambda>0$ and when $d\ge 4$, let 
$$
\begin{array}{ll}
{\mathcal A}_\Lambda:=\{(\mathbb{S}^{d-1}, [h])| &\mbox{ $h$ 
could not be joint by
a continuous path
  in the set of the metrics}\\
  &\mbox{ with positive scalar curvature to the standard metric } g_{\mathbb{S}^{d-1}} \\
& \mbox{ in the }   C^{6}(\mathbb{S}^{d-1}) \mbox{ topology}, (\mathbb{S}^{d-1}, [h]) \mbox{ is the conformal infinity } \\
& \mbox { of some CCE metric},  h\mbox{ has positive constant scalar curvature with } \\
&\|h\|_{C^{6}(g_{\mathbb{S}^{d-1}})}\le \Lambda\}

\end{array}
$$
denote the union of the path connected components of the metrics on the spheres with the positive constant scalar curvature which are not connected to the standard metric in the $C^3$ topology. 

\begin{coro} \label{gap} For any given $\Lambda>0$ and  $d\ge 4$, assume that ${\mathcal A}_\Lambda$ is not empty. Then there exists some small positive constants $\varepsilon >0 $  and $\varepsilon_1 >0 $ such that there holds
\begin{enumerate}
\item  \quad $\sup_{h\in {\mathcal A}_\Lambda} Y(\mathbb{S}^{d-1}, [h]) \leq Y(\mathbb{S}^{d-1}, [g_{\mathbb{S}^{d-1}}])-\varepsilon$; \\
\item  \quad Given any  $h\in {\mathcal A}_\Lambda$, let $(X, \p X=\mathbb{S}^{d-1}, g^+)$ be some CCE metric with conformal infinity $[h]$ on sphere $\mathbb{S}^{d-1}$. Then we have
$$
\int_{X^n} (|W|^{d/2}dvol)[g^+] > \varepsilon_1.
$$
\end{enumerate}
\end{coro}

\begin{proof}[Proof of Corollary \ref{gap}] 
We will prove this by contradiction. Suppose there exists a sequence of CCE metrics $(X,g_i^+)$ with  $[{h}_i]\in {\mathcal A}_\Lambda$ such that :\\
Either\\
$$\ds Y(\partial X, [h_i]) \to Y(\mathbb{S}^{d-1}, [g_{\mathbb{S}}]), $$
or \
$$\ds\int_{X^d} (|W|^{d/2}dvol)[g^+_i]\to 0.$$
In view of Theorem \ref{maintheorem3}, up to a subsequence, ${h}_i$ converges to the standard metric $h_{\S^{d-1}}$ in the $C^{3,\alpha}$ topology for all $\alpha\in (0,1)$ so that ${h}_i$ should be in the same connected component of metrics on the standard sphere with positive scalar curvature. Thus, we get a desired contradiction due to the definition of ${\mathcal A}_\Lambda$. 
\end{proof}

\begin{RK}
In the above result, we can assume  the metrics in the set ${\mathcal A}_\Lambda$ are in the $C^{5,\gamma}$-Cheeger-Gromov topology.
\end{RK}

\begin{appendices}
\section{Proof of Lemma \ref{bdy}}
(1) follows as
we have $\frac{R}{2(d-1)}= \frac{\hat{R}}{(d-2)}$  on the boundary (see \cite[section 6]{casechang}). \\
For (2), first we have the Gauss-Codazzi equations
$$
R_{\alpha\beta\gamma\delta}=\hat {R}_{\alpha\beta\gamma\delta} \mbox{ and } R_{1\beta\gamma\delta}=0, 
$$
so that 
$$
R_{1\alpha }=0.
$$
To prove the rest of the assertions in (2),  we write $g_1=r^2 g^+$ the compactified metric under some geodesic defining function $r$ and $g^*=\rho^2 g^+$ the  corresponding adapted metric. We  know both $g_1$ and $g^*$ have the same boundary metric $h$ which is totally geodesic. We write $g^*=w^{-2}g_1:=(\frac{r}{\rho})^{-2}g_1$. Thus, on the boundary $\p X$, we have $w\equiv 1$,  and $\nabla w\equiv 0$. As a consequence, we infer that on the boundary
$$
A_{\alpha\beta}[g^*]=A_{\alpha\beta}[g_1], \;\;W[g^*]=W[g_1],
$$
since $\nabla_\alpha\nabla_\beta w=0$ on $M$. 

We now study the Schouten tensor and Weyl tensor for the compactified metric $g_1$.   We note the full indices $i,j,k\in\{1,\cdots, d\}$. As before, we have $ r^2 g^{+} =:  g = ds^2 + g_r$, $g_r = h + g^{(2)} r^2 + O (r^4)$, $g^{(2)}_{\alpha \beta} = - \hat A_{\alpha\beta}$ where $ \hat A$ is the Schouten tensor of the metric $h$ (see \cite{G00}).

The proof to verify the rest of the assertions in (2) has done before in Section 2 of \cite{CG} when $d=4$. The proof we will present below are relatively routine, we sketch the proof here just for the convenience of the readers.  Let $(x_1, x_2, \cdots, x_d)$ denote the Fermi coordinates. We have $g_{11}=1$, $g_{1\alpha}=0$ and $g_{\alpha\beta}= h_{\alpha\beta}+O(r^2)$. A direct calculation leads to the  Christoffel symbols $\Gamma^i_{j1}=0$ on the boundary $M$, that is $({\nabla_g})_{\frac{\partial}{\partial x_\alpha}}\frac{\partial}{\partial x_1}=0 $ on the boundary $M$ 
due to the fact that the boundary is totally geodesic.

Fix a point $P$ on the boundary $M$. 
At $P$, we have the Christoffel symbols $\Gamma^i_{jk}=0$ by choosing the normal coordinates at $P$.

 Hence, we can write at $P$
$$
{R_{ijk}}^l=\frac12g^{lm}(g_{im,kj}+ g_{jk,mi}-g_{ik,mj}- g_{jm,ki}).
$$
Thus 
\beq
\label{eqschouten}
{R_{1\alpha 1}}^\gamma=-\frac12g_{\alpha\gamma,11} =-g^{(2)}_{\alpha\gamma}= \hat A_{\alpha\gamma}.
\eeq
On the other hand, on the boundary $M$, we have also the Gauss-Codazzi equations
$$
{R_{\alpha\beta\gamma}}^\delta= {{ \hat {R}}_{\alpha\beta\gamma} }^{  \;\;\;\;\; \;\;\delta}\mbox{ and }{R_{1\beta\gamma}}^\delta=0,
$$
when the boundary is totally geodesic. Therefore, at the point $P$, we have $R_{\alpha 1}=0$, $R_{\alpha\beta}=\hat R_{\alpha\beta}+{R_{\alpha1\beta}}^1$, and $R=\hat R +2R_{11}$. On the other hand, it follows from (\ref{eqschouten}) that 
$$
R_{11}=\frac{\hat R}{2(d-2)}\mbox{ and }R=\frac{d-1}{d-2} \hat R.
$$
Gathering the above relations from (\ref{eqschouten}), we infer 
$$
\begin{array}{lll}
A_{11}=0 ,\\
\ds A_{1\alpha}=\frac1{d-2}R_{1\alpha }=0 ,\\
\ds A_{\alpha\beta}=
\frac1{d-2}(\hat R_{\alpha\beta }+ {R_{1\alpha 1}}^\beta-\frac R{2(d-1)} g_{\alpha\beta})=\frac1{d-2}(\hat R_{\alpha\beta }+ \hat A_{\alpha\beta}-\frac {\hat R}{2(d-2)} g_{\alpha\beta})= \hat A_{\alpha\beta}.
\end{array}
$$
Hence, we finish the proof of (2). 

For (3), by the decomposition of Riemann curvature, we have 
$$
W_{\alpha\beta\gamma\delta}=R_{\alpha\beta\gamma\delta}-(A\owedge g)_{\alpha\beta\gamma\delta}=\hat R_{\alpha\beta\gamma\delta}-(\hat A\owedge \hat g)_{\alpha\beta\gamma\delta}=\hat W_{\alpha\beta\gamma\delta}.
$$
On the other hand, we know
$$
R_{1\beta\gamma\delta}=0=(A\owedge g)_{1\beta\gamma\delta},
$$
so that 
$$
W_{1\beta\gamma\delta}=0.
$$
Moreover, by  (\ref{eqschouten}) and the decomposition of Riemann curvature, we infer
$$
W_{1\beta1\delta}=R_{1\beta1\delta}-(A\owedge g)_{1\beta1\delta}=R_{1\beta1\delta}-A_{\beta\delta}=R_{1\beta1\delta}-\hat A_{\beta\delta}=0.
$$
It is clear that $W_{111\delta}=W_{1111}=0$. Hence, we prove (3).\\
Now, for (4), using (\ref{Cottonex}), we infer
\beq
\label{Cottonexbis}
 \rho \mathcal{C}_{ijk}=\nabla^l\rho W_{jkil}.
\eeq
Thus by taking the covariant derivative, we get
$$
\nabla^m \rho \mathcal{C}_{ijk}+ \rho \nabla^m\mathcal{C}_{ijk}=\nabla^m\nabla^l\rho W_{jkil}+\nabla^l\rho\nabla^m W_{jkil}.
$$
Hence, together with (\ref{relation5}) and by choosing $m=1$, we deduce that on the boundary $M$
$$
\mathcal{C}_{ijk}=\nabla^1 W_{jki1}=\nabla^p W_{jkip}-\nabla^\alpha W_{jki\alpha}=(d-3)\mathcal{C}_{ijk}-\hat{\nabla}^\alpha W_{jki\alpha}.
$$
That is,
$$
(d-4)\mathcal{C}_{ijk}=\hat{\nabla}^\alpha W_{jki\alpha}.
$$
Therefore
$$
(d-4)\mathcal{C}_{\beta\gamma\delta}=\hat{\nabla}^\alpha \hat{W}_{\gamma\delta\beta\alpha}=(d-4)\hat{C}_{\beta\gamma\delta}.
$$
This gives $\mathcal{C}_{\beta\gamma\delta}=\hat{C}_{\beta\gamma\delta}$ when $d\neq 4$ (when $d=4$, it is done in \cite[Lemma 2.3]{CG}). When the indices $ijk$ contain $1$, it follows from (3)
$$
\mathcal{C}_{ijk}=0.
$$
Thus we have established (4).

To see (5), using the expression of the Schouten tensor on the boundary, we have
$$
\nabla_\alpha A_{\beta\gamma}=\hat{\nabla}_\alpha \hat{A}_{\beta\gamma}, \nabla_\alpha A_{11}=\nabla_\alpha A_{1\beta}=0,
$$
which together with the expression of the Cotton tensor on the boundary, we get
$$
\nabla_1 A_{1\alpha}=0, \nabla_1 A_{\alpha\beta}=0.
$$
Applying  the second Bianchi identity, we obtain
$$
\nabla_\alpha A_{1\alpha}+\nabla_1 A_{11}=\frac{\nabla_1 R}{2(d-1)}.
$$
From this we conclude  (5).

To see (6), we first observe that the first equality in (6) is a direct result of the ones in (3).
In addition (3), we also have when the indices $ijkl$ contain 1
$$
\nabla_\alpha W_{ijkl}=0.
$$
Hence $$\nabla_1 W_{\alpha\beta\gamma1}=\nabla_k W_{\alpha\beta\gamma k}-\nabla_\delta W_{\alpha\beta\gamma\delta}= \nabla_k W_{\alpha\beta\gamma k}-\hat{\nabla}_\delta \hat{W}_{\alpha\beta\gamma\delta}=(d-3) \mathcal{C}_{\gamma\alpha\beta}-(d-4) \hat{C}_{\gamma\alpha\beta}=\hat{C}_{\gamma\alpha\beta,}$$
and
$$\nabla_1 W_{\alpha1\gamma1}=\nabla_k W_{\alpha1\gamma k}-\nabla_\delta W_{\alpha1\gamma\delta}= \nabla_k W_{\alpha1\gamma k}=(d-3)\mathcal{C}_{\gamma\alpha1}=0.$$ 
Thus we have established (6).

\section{M\"obius coordinates and weighted function spaces}\label{section 2.1}

We introduce M\"obius coordinates on conformally compact Einstein manifolds in \cite{Lee1}.\\ 

Let $(X,g^+)$ be a conformally compact Einstein $d$-manifold with a continuous conformal compactification $g=\rho^2g^+$, where $\rho$ is a defining function for $(\overline{X},g).$ For   any small postive number $\epsilon>0,$ let $X_\epsilon$ denote  the open subset of $\overline{X}$ where $0<\rho<\epsilon$ and  $\overline{X}_\epsilon$ denote the open subset where $0\leq\rho<\epsilon.$
 \par We choose smooth local coordinates $\theta=(\theta^2,\theta^2,\cdots,\theta^{d})$ on an open set $U\subset \p X.$ Extend these to coordinates
 $(\theta^1, \theta)=(\rho,\theta^2,\theta^2,\cdots,\theta^{d})$  on the open subset $\Omega= [0,\epsilon) \times U\subset \overline{X}.$ Choose finitely many $U_i$ to cover $\p X.$ The resulting coordinates on $\Omega_i=[0,\epsilon_i)\times U_i$ will be called
background coordinates for $\overline{X}.$ Let $R$ be the smallest of these $\epsilon_i,$ then any point in $\overline{X}_R$ is contained in some background coordinate chart.
 \par Now we consider the upper half-space model of hyperbolic space, i.e. $\mathbb{H}^{d}=\{(y,x)=(y,x^2,x^2,\cdots,x^{d})\in\mathbb{R}^{d}:y>0\},$ with $x^1=y$ and with the hyperbolic
metric $\check{g}$ given in coordinates by
$$\check{g}=\frac{1}{y^2}(dy^2+dx^2).$$
We let $B_1$ and $B_2$ denote the hyperbolic geodesic ball of radius 1 and 2 centered at  point $(y,x) = (1,0).$ For any point $p\in X_{R},$ let $( \rho_0,\theta_0)$ be the coordinate representation of $p$ in some fixed background chart. We can define a diffeomorphism
$\Phi_p: B_2\rightarrow X$ by
$$( \rho,\theta) = \Phi_p(y,x) = (\rho_0y,\theta_0 + \rho_0x).$$
As is shown in \cite{Lee1}, $\Phi_{p_0}$ maps $B_2$ diffeomorphically onto a neighborhood of $p_0$ in $X_R$ if $p_0\in X_{R/8}.$  And there exists a countable set of points $\{p_i\}\subset X_{R/8}$ such that the sets $\Phi_{p_i}(B_2)$ form a uniformly locally
finite covering of $X_{R/8},$ and the sets $\{\Phi_{p_i}(B_1)\}$ still cover $X_{R/8}.$  We set
$$\Phi_i=\Phi_{p_i},\ \ V_1(p_i)=\Phi_i(B_1),\ \ V_2(p_i)=\Phi_i(B_2).$$
We call $(V_2(p_i),\Phi_i^{-1})$ a M\"obius coordinate chart of $X_{R/8}.$\\

In \cite{Lee1}, Lee  introduced also the boundary M\"obius coordinates: for any given $p\in\p X,$ let $\Omega$ be a neighbourhood and $(\rho,\theta)$ be the background coordinates such that $\theta(p)=0$. For each $a>0$ and $R$ sufficiently small, we define $Y_a\subset\mathbb{H}$ and
 $Z_R(p)\subset\Omega\subset\overline{X}$:
 \begin{eqnarray} \label{B1}Y_{a} &=\{(y,x) \in \mathbb{H} :|x|<a, 0<y<a\} \\  \label{B2}Z_{R}(p) &=\{(\rho,\theta) \in \Omega :|\theta|<R, 0<\rho<R\}\end{eqnarray}
 Define a chart $\Psi_{p, R} : Y_{1} \rightarrow Z_{R}(p)$ by
 $$(x,y)\mapsto (Ry, Rx)=(\rho, \theta). $$
 We will call $\Psi_{p, R}$ a boundary M\"obius chart of radius $R$ centered at $p.$
\\

 Assume $(X,g^+)$ is a conformally compact Einstein  manifold of class $C^{l,\beta}$ with $l\ge 2$  and $0\le\beta <1$. We consider a geometric tensor bundle $E$ of weight $r$ on $\bar X$ (resp. $X$).  In \cite{Lee1}, we introduce weighted H\"older spaces of tensor fields 
 $C^{m,\alpha}_{(s)}(\overline{X};E)$  on 
 $\bar X$ with $ m + \alpha \le l +\beta $  and $ s \le l+\beta $ (resp. 
 $C^{m,\alpha}_{t}(X;E)$  on $X$ with $m+\alpha\le l+\beta$  and $t\in \R$).\\

\par There are the following relationships between the H\"older spaces on $X$ and those on $\overline{X}$:

\begin{lemm}\cite[Lemma 3.7]{Lee1}\label{Lee1bis}
Let $E$ be a geometric tensor bundle of weight $r$ over $\overline{X},$ and suppose $0<\alpha<1,0<m+\alpha\leq l+\beta,$ and $0\leq s\leq k+\alpha.$ The following inclusions are continuous.
\par (a) $C^{m,\alpha}_{(s)}(\overline{X};E)\hookrightarrow C^{m,\alpha}_{s+r}(X;E),$
\par (b) $C^{m,\alpha}_{m+\alpha+r}(X;E)\hookrightarrow C^{m,\alpha}_{(0)}(\overline{X};E).$
\end{lemm}
\end{appendices}


\begin{thebibliography}{99}



\bibitem{Anderson0} M. Anderson, {\sl Convergence and rigidity of manifolds under Ricci curvature bounds}, \newblock{Invent. Math.}, 
(1990) 102 (2), 429-445.

\bibitem{anderson1} M. Anderson, {\sl  Einstein metrics with prescribed conformal infinity on 4-manifolds}, \newblock  {Geom. Funct. Anal. }18 (2008), 305-366.



\bibitem{anderson5} M. Anderson, {\sl $L^2$ curvature and volume renormalization of the AHE metrics on 4-manifolds}, \newblock { Math. Res. Lett.}, 8 (2001) 171-188.


\bibitem{AKKLT} M. Anderson, A. Katsuda, Y. Kurylev, M. Lassas, and M. Taylor, {\sl  Boundary regularity for the Ricci equation, geometric convergence, and Gelfand's inverse boundary problem}, 
Invent. Math., 158 (2): 261-321, 2004.


 



\bibitem{Biquard} O. Biquard, {\sl  Continuation unique \`a partir de l'infini conforme pour les m\'etriques d'Einstein}, \newblock { Math. 
Res. Lett.} 15 (2008), 1091-1099.

\bibitem{Biquard1} O. Biquard, {\sl Einstein deformations of hyperbolic metrics}, surveys in differential geometry: essays on Einstein manifolds, 235-246, Surv. Differ. Geom., 6, Int. Press, Boston, MA, 1999.

\bibitem{BiquardHerzlich} O. Biquard and M. Herzlich,  {\sl  Analyse sur un demi-espace hyperbolique et poly-homog\'en\'eit\'e locale}, 
\newblock { Calc. Var. P.D.E. 51} (2014), 813-848.

\bibitem{BB} H. P. Boas and R. P.  Boas,  {\sl  Short proofs of three theorems on harmonic function}, \newblock { Proceeding of the AMS} 102 (1988), 906-908.


\bibitem{B} T. Branson,  {\sl  Differential operators canonically associated to a conformal structure}, \newblock { Math. Scand. 57 (1985)}, 293-345.



\bibitem{casechang} J. Case and S.-Y. A. Chang,  {\sl  On fractional GJMS operators},  \newblock { Comm. Pure Appl. Math.} 69 (2016), 
1017-1061.



\bibitem{CG} S.-Y. A. Chang and Y. Ge, {\sl  Compactness of conformally compact Einstein manifolds in dimension 4},   \newblock { Adv. Math.} 340 (2018), 588-652.

\bibitem{CGQ} S.-Y. A. Chang, Y. Ge and J. Qing, {\sl  Compactness of conformally compact Einstein manifolds in dimension 4 II}, \newblock { Adv. Math.} 373(2020), 107325.

\bibitem{CGY}  S.-Y. A. Chang, M.  Gursky and P.  Yang,, {\sl   A conformally invariant sphere theorem in four dimensions},  \newblock { Publ. Math. Inst. Hautes \'Etudes Sci }98(2003), 105-143.


\bibitem{CYa} S.-Y. A. Chang and  R. Yang,  {\sl   On a class of non-local operators in conformal geometry},  \newblock { Chinese Annals of Mathematics} 38 (2017), 215-234.

\bibitem{CGT} J. Cheeger, M. Gromov and M. Taylor, {\sl  Finite propagation speed, kernel estimates for functions of the Laplace operator, 
and the geometry of complete Riemannian manifolds}, J. Differential Geom. 17 (1982), 15-53.


\bibitem{CDLS} P. T. Chru\'sciel, E. Delay, J. M. Lee, D. N. Skinner, \newblock {\em  Boundary regularity of
            conformally compact Einstein metrics}, {J. Differential Geom.} 69 (2004), 111-136.


\bibitem{Dutta} S. Dutta and M. Javaheri, \newblock {\em Rigidity of conformally compact manifolds with the round sphere as the conformal infinity}, \newblock { Adv. Math.} 224 (2010), 525-538.

\bibitem{FG} C. Fefferman and C. R. Graham, \newblock {\em  $Q$-curvature and Poncar\'e metrics}, \newblock {Math. Res. Lett.} 9 (2002), 139-151.

\bibitem{FG12} C. Fefferman and C. R. Graham,
\newblock {\em  The ambient metric},  Annals of Mathematics Studies, 178, Princeton University Press, Princeton, (2012).

\bibitem{GHo} D. Gilbarg and L. H\"ormander,   {\sl Intermediate schauder estimates},  \newblock { Archive
for Rational Mechanics and Analysis}, 74 (1980) 297-318.

\bibitem{GT} D. Gilbarg and N. Trudinger,  {\sl Elliptic partial differential equations of second order}, \newblock { Grundlehren der mathematischen Wissenschaften} 224, Ed. Springer, (2001).


\bibitem{G00} C. R. Graham,
\newblock {\em  Volume and Area renormalizations for conformally compact Einstein metrics}, \newblock The Proceedings of the 19th Winter School "Geometry and Physics" (Srn\`{i}, 1999). 
\newblock Rend. Circ. Mat. Palermo 63  (2000) 31-42.

\bibitem{G09} C. R. Graham,
\newblock {\em  Extended obstruction tensors and renormalized volume coefficients},   \newblock { Adv. Math.} 220 (2009), 1956-1985.

\bibitem{GH} C. R. Graham and K. Hirachi,
\newblock {\em  The ambient obstruction tensor and $Q$-curvature,}, \newblock In AdS/CFT correspondence: Einstein metrics and their conformal boundaries, 
\newblock volume 8 of IRMA lec. Math. Theor. Phys., pages 59-71, Eur.Math. Soc., Z\"urich, 2005.



\bibitem{GL} C. R. Graham and J. Lee, {\em Einstein metrics with prescribed conformal infinity on the ball.}  {Adv. Math.} 87 
(1991), no. 2, 186 - 225.




\bibitem{GZ} C. R. Graham and M. Zworski,
\newblock {\em  Scattering matrix in conformal geometry},  {Invent. Math.} 152 (2003), 89-118.


\bibitem{Helli} D. W. Helliwell, {\sl Boundary regularity for conformally compact Einstein metrics in even dimensions}, \newblock 
{Communications in Partial Differential Equations}, 33(5), (2008), 842 - 880. 

\bibitem{Kodani} S. Kodani, {\sl  Convergence theorem for Riemannian manifolds with boundary}, Compositio Math., 75(2):171 - 192, 1990.

\bibitem{Knox} K. Knox,  {\sl  A compactness theorem for riemannian manifolds with boundary and applications} arXiv:1211.6210 [math.DG].


\bibitem{Lee1} J. Lee, {\sl Fredholm operators and Einstein metrics on conformally compact manifolds}, \newblock{Mem. Amer. Math. 
Soc.} 183 (2006), no. 864, vi+83 pp.

\bibitem{LQS} G. Li, J. Qing and Y. Shi, \newblock {\em  Cap phenomena and curvature estimates for conformally compact Einstein manifolds },
\newblock {Trans. Amer. Math. Soc. 369 (2017), no. 6, 4385 - 4413.} 



\bibitem{Mald-1} J. Maldacena, {\sl The large N limit of superconformal field theories and supergravity}, 
\newblock{ Adv. Theo. Math. Phy.} {\textbf  2} (1998) 231-252, hep-th/9711200.

\bibitem{Mald-2} J. Maldacena, TASI 2003 Lectures on AdS/CFT, hep-th/0309246.

\bibitem{Mald} J. Maldacena, {\sl Einstein gravity from conformal gravity},  arXiv:1105.5632.

\bibitem{MM} R. Mazzeo and R.  Melrose, {\sl  Meromorphic extension of the resolvent on complete spaces with asymptotically constant negative curvature}, \newblock { J. Funct. Anal.} 113 (1991),  25-45 .

\bibitem{P}  S. Paneitz, {\sl   A quartic conformally covariant differential operator for arbitrary pseudo-Riemannian manifolds}, Preprint (1983), arXiv:0803.4331.

\bibitem{Perales} R. Perales, {\em A survey on the convergence of manifolds with boundary}, \newblock {Contemp. Math.} 657 (2016), 179-188.


\bibitem{Q} J. Qing, {\sl   On the rigidity for conformally compact Einstein manifolds}, IMRN Volume 2003, Issue 21,  1141-1153.


\bibitem{Wi} E.Witten, {\em   Anti de Sitter space and holography},  Adv.Theor.Math.Phys., {\textbf 2} (1998), 253-291.


\bibitem{Wong} J. Wong, {\sl   An extension procedure for manifolds with boundary}, \newblock {Pacific J. Math.} 235 (2008), 173-199.

\end{thebibliography}
\end{document}